\documentclass[11pt]{amsart}
\usepackage[margin=1in]{geometry}

\usepackage{hyperref}
\usepackage[table]{xcolor}
\usepackage{listings}
\usepackage{graphicx}
\usepackage{float}
\usepackage{booktabs}
\usepackage{verbatim}
\usepackage{xcolor}
\usepackage{amsmath}
\usepackage{amssymb}
\usepackage{amsthm}
\usepackage{amsfonts}
\usepackage{float}
\usepackage{wrapfig}
\usepackage{array}
\usepackage{thm-restate}
\usepackage{tikz, tikz-cd}
\usetikzlibrary{decorations.markings, patterns, decorations.pathmorphing}
\restylefloat{figure}

\newcommand{\C}{\mathbb{C}}

\newcommand{\Z}{\mathbb{Z}}

\newcommand{\R}{\mathbb{R}}

\newcommand{\tr}{\text{tr}}

\newcommand{\Aff}{\text{Aff}^+}

\newtheorem{thm}{Theorem}[section]
\newtheorem{prop}[thm]{Proposition}
\newtheorem{lem}[thm]{Lemma}

\theoremstyle{definition}

\theoremstyle{remark}
\newtheorem{remark}[thm]{Remark}

\tikzset{mid/.style={
        decoration={markings,
            mark= at position 0.5 with {\arrow{#1}} ,
        },
        postaction={decorate}
    }
}

\tikzset{m1/.style={
        decoration={markings,
            mark= at position 0.25 with {\arrow{#1}} ,
        },
        postaction={decorate}
    }
}

\tikzset{m3/.style={
        decoration={markings,
            mark= at position 0.75 with {\arrow{#1}} ,
        },
        postaction={decorate}
    }
}

\begin{document}

\title{The realization problem for dilation surfaces}

\author{Jane Wang}
\address{Department of Mathematics,  Indiana University Bloomington, 831 E. 3rd St, Bloomington, IN 47401}
\email{wangjan@iu.edu}

\begin{abstract}
Dilation surfaces, or twisted quadratic differentials, are variants of translation surfaces. In this paper, we study the question of what elements or subgroups of the mapping class group can be realized as affine automorphisms of dilation surfaces. We show that dilation surfaces can have exotic Dehn twists in their affine automorphism groups and will establish that only certain types of mapping class group elements can arise as affine automorphisms of dilation surfaces. We also generalize a construction of Thurston that constructs a translation surface from a pair of filling multicurves to dilation surfaces. This construction will give us dilation surfaces that realize a pair of Dehn multitwists in their affine automorphism groups. 
\end{abstract} 
\maketitle

\section{Introduction} 
\label{sec:intro}

The main objects of study in this paper are \emph{dilation surfaces}, also known as \emph{twisted quadratic differentials}. There are a few ways of thinking about a dilation surface, including as a collection of polygons with sides identified by translation, dilation, and $180^\circ$ rotation. Dilation surfaces are geometric surfaces that locally look like Euclidean space except at finitely many cone points that have angle $\pi n$ around them. To each dilation surfaces, there is an associated holonomy representation $\rho : \pi_1(X) \rightarrow \R_+$ that records how the local Euclidean metric scales when moving around loops on the surface. Given this holonomy representation, we can also think of a dilation surfaces as a triple $(X, \rho, q)$ of a Riemann surface $X$, a holonomy representation $\rho$, and a quadratic differential $q$ twisted by the holonomy $\rho$. 

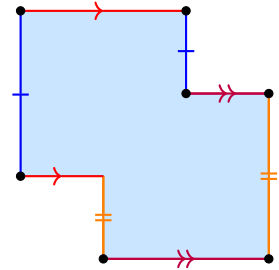
\begin{wrapfigure}{R}{0.25\textwidth}
	\begin{center}
	\begin{tikzpicture}[scale = 1.1]
	\draw[fill={rgb,255:red,202; green,229; blue,255}] (1,0) -- (3,0) -- (3,2) -- (2,2) -- (2,3) -- (0,3) -- (0,1) -- (1,1) -- (1,0); 
	
	\draw[red, thick,->] (0,3) -- (1,3);
	\draw[red, thick] (1,3) -- (2,3);
	
	\draw[blue, thick,-|] (2,3) -- (2,2.5);
	\draw[blue, thick] (2,2.5) -- (2,2);
	
	\draw[purple, thick] (2,2) -- (3,2);
	\draw[purple, thick,->] (2,2) -- (2.5,2);
	\draw[purple, thick,>-] (2.5,2) -- (3,2);
	
	\draw[orange, thick] (3,2) -- (3,0);
	\draw[orange, thick,-|] (3,2) -- (3,1.02);
	\draw[orange, thick,|-] (3,0.98) -- (3,0);
	
	\draw[purple, thick] (1,0) -- (3,0);
	\draw[purple, thick,->] (1,0) -- (2,0);
	\draw[purple, thick,>-] (2,0) -- (3,0);

	\draw[orange, thick] (1,0) -- (1,1);
	\draw[orange, thick,-|] (1,0) -- (1,0.48);
	\draw[orange, thick,|-] (1,0.52) -- (1,1);
	
	\draw[red, thick,->] (0,1) -- (0.5,1);
	\draw[red, thick] (0.5,1) -- (1,1);
	
	\draw[blue, thick,-|] (0,1) -- (0,2);
	\draw[blue, thick] (0,2) -- (0,3);
	\draw[fill] (1,0) circle [radius = 0.05]; 
	\draw[fill] (3,0) circle [radius = 0.05]; 
	\draw[fill] (3,2) circle [radius = 0.05]; 
	\draw[fill] (2,2) circle [radius = 0.05]; 
	\draw[fill] (2,3) circle [radius = 0.05]; 
	\draw[fill] (0,3) circle [radius = 0.05]; 
	\draw[fill] (0,1) circle [radius = 0.05]; 
	\end{tikzpicture}
	\end{center}
	\caption{A genus $2$ dilation surface with one cone point.}
\end{wrapfigure}

The focus of this paper will be to study the \emph{affine automorphism group} $\Aff(X, \rho, q)$ of a dilation surface $(X, \rho, q)$, as well as the stabilizer $\text{SL}(X, \rho, q)$ of the surface under the $\text{SL}(2,\mathbb{R})$ action, also known as the \emph{Veech group} of the surface. Dilation surfaces are relatively new objects of study, and as such, there is much still unknown about their geometry and their dynamics. Much of the current work on dilation surfaces is motivated by what is known about \emph{translation surfaces}, their well-studied untwisted counterparts. One reason we study the Veech groups of dilation surfaces is because understanding the Veech groups of translation surfaces has many applications to mathematical billiards as well as to the study of curves in the moduli space of Riemann surfaces. For example, there has been a lot of work done to find translation surfaces with Veech groups that are a lattice in $\text{SL}(2,\mathbb{R})$, since in this case the surface is known to have optimal dynamics. 

In contrast, very little is known about the geometry and dynamics of dilatin surfaces and there has been relatively little work done in understanding dilation surfaces and their affine automorphism groups. Perhaps surprisingly, it is not even known if a dilation surface with nontrivial holonomy can have a Veech group that is a lattice subgroup of $\text{SL}(2,\mathbb{R})$. \\

\noindent \textbf{The Realization Problem.} Partially motivated by this problem of determining whether dilation surfaces can have lattice Veech group, in this paper we will study the Veech groups of dilation surfaces and what mapping class group elements can arise as affine automorphisms of dilation surfaces. We will primarily be concerned with the following question, which we will call the realization problem. 

\textbf{Realization Problem:} What elements or subgroups of the mapping class group $\text{Mod}(\Sigma)$ can be realized in $\Aff(X, \rho, q)$ of some dilation surface? 

When we restrict our attention in the realization problem to individual elements of the mapping class group, we are asking the natural question of what topological maps can be realized as geometric maps on dilation surfaces. When we instead think about subgroups of the mapping class group, we can ask if finding large realizable subgroups gets us a step closer to finding dilation surfaces with lattice Veech group. 

As methods and results for translation surfaces sometimes motivate the study of similar problems for dilation surfaces, let us first survey what is known about the realization problem for translation surfaces. In particular, what mapping class group elements can be realized in $\Aff(X, q)$ for a translation surface $(X, q)$? It is well-known that the trace of an affine automorphism of a translation surfaces determines the type of the underlying topological map (e.g. see \cite{Mo}). If $\varphi \in \Aff(X, q)$ and $D\varphi$ is its derivative matrix, then $\varphi$ is finite order if and only if $D\varphi$ is elliptic and has $|\tr(D\varphi)| < 2$, $\varphi$ is pseudo-Anosov if and only if $D\varphi$ is hyperbolic and has $|\tr(D\varphi)| > 2$, and $\varphi^n$ is a Dehn multitwist for some $n$ if and only if $D \varphi$ is parabolic and has $|\tr(D\varphi)| = 2$. 

It is natural then to wonder if a similar result for dilation surfaces can help us to resolve the realization problem. It is here that we meet our first difficulty in solving the realization problem. While Dehn twists and multitwists $\varphi$ represented in the affine autmorphism group of translation surfaces always arise as twists around Euclidean cylinders on the surface and thus have that $|\tr(D\varphi)| = 2$, the same no longer holds for dilation surfaces. Dilation surfaces can have both \textit{standard Dehn multitwists}, Dehn multitwists around Euclidean cylinders, as well as \textit{exotic Dehn multitiwsts}, Dehn multitwists around affine cylinders whose core curves have holonomy, in their affine automorphism groups. 

\begin{restatable*}{thm}{dehn} 
	\label{thm:dehn}
	For every $g \geq 3$, there exists a dilation surface of genus $g$ with an exotic Dehn multitwist in its affine automorphism group. 
\end{restatable*}

These exotic Dehn multitwists $\varphi$ will have that $|\tr(D\varphi)| > 2$, disrupting the one-to-one correspondence that exists for the translation surfaces between what type of map $\varphi \in Aff(X, \rho, q)$ is topologically and whether $D\varphi$ is elliptic, parabolic, or hyperbolic. 

Since the trace of an affine automorphism of a translation surface determines what type of map it is, we can see that reducible mapping class group elements that exhibit mixed behavior (e.g. being pseudo-Anosov on one part of the surface and a Dehn twist on another) cannot be represented in $\Aff(X,q)$ of any translation surface. With the existence of exotic Dehn multitwists on dilation surfaces having the same trace behavior as pseudo-Anosov maps, it is no longer clear if mixed behavior reducible mapping class group elements can be represented in $\Aff(X, \rho, q)$ for dilation surfaces. Nevertheless, we can show with different techniques that a similar classification of representable mapping class group elements as for translation surfaces also exists for dilation surfaces. 

\begin{restatable*}{thm}{classification} 
	\label{thm:classification_dilation}
	Let $\varphi$ is in $\Aff(X, \rho, q)$ and let $D\varphi \in \text{SL}(2,\mathbb{R})$. If $D\varphi$ is elliptic, then $\varphi$ is finite order. If $D \varphi$ is parabolic, then some power $\varphi^n$ of $\varphi$ is a Dehn multitiwst. If $D \varphi$ is hyperbolic, then $\varphi$ is either pseudo-Anosov or some power $\varphi^n$ of $\varphi$ is a Dehn multitwist. 
\end{restatable*}

This theorem gives us necessary conditions for a mapping class group element to be realizable in $\Aff(X, \rho, q)$ for some dilation surfaces. We will also address sufficient conditions for finite order and pseudo-Anosov mapping class group elements to be realizable in $\Aff(X, \rho, q)$ for some dilation surface after fixing the holonomy $\rho$. 

Having mostly resolved the realization problem for individual elements of the mapping class group, we can then move on to the realization problem for subgroups of the mapping class group. The next result is a generalization of the Thurston construction for constructing pseudo-Anosov maps on translation surfaces from pairs of filling multicurves on a surface (see \cite{FM}, Section 14.1 for a reference for this construction). It also gives us dilation surfaces with large subgroups of $\text{SL}(2,\mathbb{R})$ as their Veech group. Specifically, from this construction we obtain dilation surfaces with $\Z*\Z$, the free group on two generators, as a subgroup of the Veech group. 

First, we present a few definitions. We say that a collection of not necessarily disjoint simple closed curves \emph{fills} a surface, if the complement of these curves is a collection of topological disks. A multicurve $\gamma = \{\gamma_i\}$ is a collection of disjoint simple closed curves, up to isotopy, and a Dehn multitwist $T_\gamma$ is the product of Dehn twists along each $\gamma_i$. Then, we have the following theorem.

\begin{restatable*}{thm}{thurston}
	\label{thm:thurston}
	
	Given Dehn multitwists $T_\alpha$ and $T_\beta$ along multicurves $\alpha = \{\alpha_i\}$ and $\beta = \{\beta_i\}$ such that $\alpha \cup \beta$ fill the surface, if $\rho : \pi_1(\Sigma) \rightarrow \R_+$ is trivial on each $\alpha_i$ and $\beta_i$, then there exists a dilation surface $(X, \rho, q)$ with $T_\alpha$ and $T_\beta$ represented in $\Aff(X, \rho, q)$ as standard Dehn multitwists. 
\end{restatable*}

A direct consequence of this theorem is that for pseudo-Anosov maps $\psi$ resulting from the product of the Dehn multitwists $T_\alpha$ and $T_\beta$, this gives another way of constructing dilation surfaces with $\psi$ represented in $\Aff(X, \rho, q)$. This process is also an example of a more general construction that produces dilation surfaces and twisted quadratic differentials from pairs of filling multicurves. 

A natural question to ask here is whether one can find dilation surfaces with nontrivial holonomy and lattice Veech groups with this construction. Unfortunately, we will see that the answer is no, or at least not easily. The elements $D(T_\alpha)$ and $D(T_\beta)$ found with this construction will never be enough to generate a lattice subgroup of $\text{SL}(2,\mathbb{R})$. However, there is still hope that one of the dilation surfaces $(X, \rho, q)$ found via this construction could nevertheless have a lattice Veech group once we include all of the elements of $\text{SL}(X, \rho, q)$ and not just $D(T_\alpha)$ and $D(T_\beta)$. 

\textbf{Related references.} We comment here that there has been much recent interest in dilation surfaces. Most closely related to this paper is a paper by Duryev, Fourgeron, and Ghazouani (\cite{DFG}) that studies the Veech groups of dilation surfaces. Various other authors have also analyzed the dynamics of the straight-line flow on dilation surfaces (\cite{BFG}, \cite{BG}, \cite{BS}, \cite{G}, \cite{HW}). There have also been older works studying objects related to dilation surfaces. In \cite{V4}, Veech studied spaces of complex affine surfaces, which are cousins to dilation surfaces. The directional foliations of dilation surfaces are twisted measured foliations, also called affine foliations or laminations. Hatcher and Oertel study the space of twisted measured foliations in \cite{HO}. There is also a connection between dilation surfaces and affine interval exchange maps. The latter topic is studied by various authors including Marmi, Moussa, and Yoccoz in \cite{MMY}.

\textbf{Structure of the paper.} In Section \ref{sec:background}, we will cover background material dilation surfaces and their affine automorphism groups and Veech groups. The statement of the realization problem is then given in Section \ref{sec:realization}. We will also give an overview of the situation for genus $1$ surfaces before showing the existence of exotic Dehn twists in higher genus and proving that the types of mapping class group elements that can be realized in $\Aff(X, \rho, q)$ are limited. Then, in Section \ref{sec:elements}, we will see that we can always realize certain mapping class group elements as affine automorphism and in Section \ref{sec:multicurves}, we will show by generalizing a construction of Thurston that some special subgroups of the mapping class group generated by Dehn multitwists can be realized in $\Aff(X, \rho, q)$. 

\section{Background}
\label{sec:background}

\subsection{From translation surfaces to dilation surfaces} 

The objects of study in this paper, dilation surfaces, are in a way a generalization of translation surfaces, their more well-studied cousins. A \textbf{translation surface} can be defined geometrically as a collection of polygons with sides identified in parallel opposite pairs by translation, up to cut and paste equivalence. Equivalently, they can be defined as a pair $(X, \omega)$ where $\omega$ is a holomorphic one-form (a section of $K_X$, the cotangent bundle) on the Riemann surface $X$, and have moduli spaces called \textbf{strata} that are related to the moduli spaces of Riemann surfaces. A variant that people study are half-translation surfaces, which are a pair $(X, q)$ where $q$ is a quadratic differential (a meromorphic section of the cotangent bundle squared with at most simple poles). Geometrically, half-translation surfaces can be defined as a collection of polygons with sides identified by translation and rotation by $180$ degrees, up to cut and paste equivalence. 

Translation surfaces and half-translation surfaces have a flat geometry away from finitely many cone points, a flat metric coming from the holomorphic one-form or quadratic differential, and a directional flow or foliation  on the surface (up to $180^\circ$ rotation for half-translation surfaces). Many interesting dynamical questions can be asked and answered about the directional flow on a translation surface. Some of these questions involve either the \textbf{saddle connections} (straight line trajectories connecting two cone points) or \textbf{separatrices} (straight line trajectories emanating from a cone point) of the surface. Translation surfaces also arise naturally as useful tools in the study of mathematical billiards and have been used successfully to show, for example, that billiards in certain polygonal tables have optimal dynamics (\cite{V3}) and to make progress on the illumination problem (\cite{LMW}). 

There is also interesting dynamics on the strata of translation surfaces. There is a natural $\text{SL}(2,\mathbb{R})$ action on each stratum, and the closures of the orbits of this action are known to be algebraic varieties. The \textbf{Veech group} is the stabilizer of the translation surface under the $\text{SL}(2,\mathbb{R})$ action and
 sometimes reveals properties of the straight-line flow on the underlying flat surface. In the Teichm\"{u}ller space of Riemann surfaces, the $\text{SL}(2,\mathbb{R})$ orbits of translation surfaces generate isometrically immersed hyperbolic disks. The study of translation surfaces, their dynamics, and their moduli space is a rich and active area of current research that both draws on tools from many different fields and offers a different perspective to study areas ranging from dynamics, to algebraic geometry and number theory, to geometry and topology. For a more extensive introduction to translation surfaces, see \cite{Wr} or \cite{Z}. 

Dilation surfaces are the broadest class of objects that maintains the key properties of translation surfaces including the directional flow on the surface and the $\text{SL}(2,\mathbb{R})$ action. By broadening the field of study to dilation surfaces, we see new geometric and dynamical behaviors that do not occur for translation surfaces.

Like translation surfaces, \textbf{dilation surfaces} can be defined in two ways: 

\begin{enumerate} 
	\item[(1)] A triple $(X, \rho, q)$, where $X$ is a Riemann surface and $\rho : \pi_1(X) \rightarrow \R_+$ is a \textbf{holonomy representation}. If we define $L_\rho$ as a flat complex line bundle over $X$ with holonomy $\rho$, then $q$ is a meromorphic section of $K_X^2 \otimes L_\rho^2$ with at most simple poles, where $K_X$ is the cotangent bundle. We call $q$ a \textbf{twisted quadratic differential}.
	
	\item[(2)] A collection of polygons with sides identified in parallel opposite pairs by maps in the group $\{f(z) = az + b : a \in \R_{\neq 0}, b \in \C\}$, such that the holonomy around each cone point is trivial. 
\end{enumerate}

We will take a more in-depth look at these definitions after the following remark and example.  

\begin{remark} We remark that there is some discrepancy in the literature on the definition of a dilation surface. Some authors prefer to define dilation surfaces with twisted holomorphic one-forms rather than quadratic differentials. Some other authors do not mandate that dilation surfaces have trivial holonomy around singularities. We choose to work with twisted quadratic differentials because the dimension vector space of such differentials $QD_\rho(X)$ is only dependent on genus, which is not true for twisted holomorphic one-forms. One can see this via a Riemann Roch computation. As for the holonomy condition, we impose trivial holonomy around cone points so that the associated twisted measured foliations and laminations coming from the directional foliations on the surface remain locally nice and so cone points correspond to zeros or poles of $q$. 
\end{remark}

Figure \ref{fig:genus2} shows an example of a genus $2$ dilation surface. Sides with identical markings are identified by translation and dilation. We notice that the vertices are all identified to one point that is a cone point of cone angle $6 \pi$, and that every other point on the surface has a locally flat neighborhood. If we travel around loops on the surface, lengths and areas scale by some holonomy. However there is trivial holonomy around small loops around cone points. 

\begin{figure}[ht]
	\begin{center}
		\begin{tikzpicture}
		\draw[fill={rgb,255:red,202; green,229; blue,255}] (1,0) -- (3,0) -- (3,2) -- (2,2) -- (2,3) -- (0,3) -- (0,1) -- (1,1) -- (1,0); 
		
		\draw[red, thick,->] (0,3) -- (1,3);
		\draw[red, thick] (1,3) -- (2,3);
		
		\draw[blue, thick,-|] (2,3) -- (2,2.5);
		\draw[blue, thick] (2,2.5) -- (2,2);
		
		\draw[purple, thick] (2,2) -- (3,2);
		\draw[purple, thick,->] (2,2) -- (2.5,2);
		\draw[purple, thick,>-] (2.5,2) -- (3,2);
		
		\draw[orange, thick] (3,2) -- (3,0);
		\draw[orange, thick,-|] (3,2) -- (3,1.02);
		\draw[orange, thick,|-] (3,0.98) -- (3,0);
		
		\draw[purple, thick] (1,0) -- (3,0);
		\draw[purple, thick,->] (1,0) -- (2,0);
		\draw[purple, thick,>-] (2,0) -- (3,0);

		\draw[orange, thick] (1,0) -- (1,1);
		\draw[orange, thick,-|] (1,0) -- (1,0.48);
		\draw[orange, thick,|-] (1,0.52) -- (1,1);
		
		\draw[red, thick,->] (0,1) -- (0.5,1);
		\draw[red, thick] (0.5,1) -- (1,1);
		
		\draw[blue, thick,-|] (0,1) -- (0,2);
		\draw[blue, thick] (0,2) -- (0,3);
		\draw[fill] (1,0) circle [radius = 0.05]; 
		\draw[fill] (3,0) circle [radius = 0.05]; 
		\draw[fill] (3,2) circle [radius = 0.05]; 
		\draw[fill] (2,2) circle [radius = 0.05]; 
		\draw[fill] (2,3) circle [radius = 0.05]; 
		\draw[fill] (0,3) circle [radius = 0.05]; 
		\draw[fill] (0,1) circle [radius = 0.05]; 
		\end{tikzpicture}
	\caption{A genus $2$ dilation surfaces with one cone point of cone angle $2\pi$.}
	\label{fig:genus2}
	\end{center}
\end{figure}
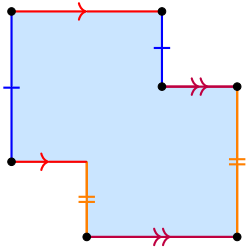

Let us examine the objects in the above definition  of a dilation surface more in depth. On a dilation surface, lengths and areas cannot be globally defined. The map $\rho : \pi_1(X) \rightarrow \R_+$ is called the holonomy representation and will dictate how lengths and areas change as one travels around loops on the surface. We can then construct the bundle $L_\rho$ to be a line bundle that encodes the twisting data of $\rho$. To do this, we let $\widetilde{X}$ be the universal cover of $X$. There is an action of $\pi_1(X)$ on both $\widetilde{X}$ and $\C$. If $\gamma \in \pi_1(X)$, then $\gamma$ acts on $\widetilde{X}$ by the corresponding deck transformation and on $\C$ by multiplication by $\rho(\gamma)$. Then, we can define $L_\rho$ as $\widetilde{X} \times \C / \sim$ where the equivalence $\sim$ is defined by $( x, z) \sim (\gamma(x), \rho(\gamma) z)$. 

Given a meromorphic section $q$ of $K_X^2 \otimes L_\rho^2$ with zeros of arbitrary order but at most simple poles, we can see that away from the zeros and poles, $q$ locally defines a flat structure given by the coordinate $z = \int q^{1/2}$. This coordinate is well-defined up to translation, rotation by $180$ degrees, and dilation by nonzero real factors. Around the zeros of order $n$ of $q$, we see cone angles of $(n+2)\pi$. Around poles of order $1$, we see a cone angle of $\pi$. Piecing together these charts, we imbue the underlying surface with a local flat geometry. It takes a little more work to realize that we can triangulate the surface and therefore decompose it into polygons with sides identified by maps in $\{f(z) = az + b: a \in \R_{\neq 0}, b \in \C\}$. Conversely, given such a collection of polygons in $\C$, we can locally define the quadratic differential $dz^2$ coming from the $z$ coordinate in $\C$. These local forms piece together to form a holomorphic section of $K_X^2 \otimes L_\rho^2$.

A simple example of a dilation surface where we can easily see that associated twisted quadratic differential is a \textbf{Hopf torus}. Specifically, we will start with $\C^*$ and the quadratic differential $dz^2$ and quotient out by the action of $z \mapsto 2z$, to get an annulus with the outer ring identified to the inner ring by a dilation. 
	
	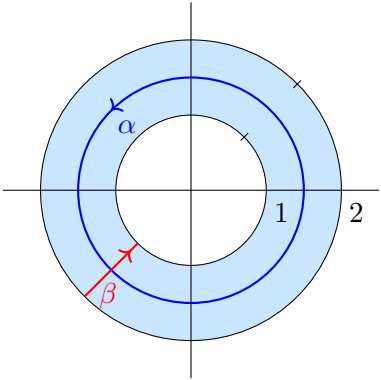
\begin{figure}[ht]
	\begin{center} 
		\begin{tikzpicture}
		\draw[fill = {rgb,255:red,202; green,229; blue,255}] (0,0) circle [radius=2];
		\draw[fill = white] (0,0) circle [radius=1];
		\draw (-2.5,0) -- (2.5,0);
		\draw (0,-2.5) -- (0,2.5); 
		\node at (1.2, -0.3) {$1$};
		\node at (2.2, -0.3) {$2$};
		\draw (0.66, 0.66) -- (.76, .76); 
		\draw (1.36, 1.36) -- (1.46, 1.46); 
		
		\draw[thick, color = blue, 
		decoration={markings, mark=at position 0.375 with {\arrow{>}}},
		postaction={decorate} 
		] (0,0) circle [radius = 1.5];
		\draw[thick, color = red, decoration={markings, mark=at position 0.25 with {\arrow{<}}},
		postaction={decorate} ] (-0.71,-0.71) -- (-1.41, -1.41); 
		
		\draw[color = blue] node at (-0.85, 0.85) {$\alpha$};
		\draw[color = red] node at (-1.1, -1.4) {$\beta$};
		\end{tikzpicture} 
	
	\end{center} 
	\caption{A Hopf torus. The outer and inner rings of the annulus are identified by a $\times 2$ map.}
	\label{fig:hopf}
	\end{figure}
	
	If $\alpha$ and $\beta$ are generators of $\pi_1(X)$ as shown in Figure \ref{fig:hopf}, then $\rho (\alpha ) = 1$ and $\rho(\beta) = 2$. The dilation surface pictured then corresponds to the Riemann surface that is the annulus with the outer and inner boundaries identified, with the section of $K_X^2 \otimes L_\rho^2$ coming from $dz^2$ on $\C^*$. We notice that if we start with a form that is $dz$ in local coordinates on the surface and follow it one time around $\beta$, the way that the sides are identified by scaling also scales the form by $2$. 
	
We would also like to understand when two dilation surfaces defined via either definition are equivalent. 

\begin{enumerate}
	\item Two triples $(X, \rho, q)$ and $(X^\prime, \rho^\prime, q^\prime)$ define the same dilation surface if there exists a holomorphic map $f : X \rightarrow X^\prime$ such that $f^* (\rho^\prime) = \rho$, and there exists a lift $F : L_\rho \xrightarrow{\cong} L_{\rho^\prime}$ that preserves the real subspaces of fibers $\R \subset \C$ such that $f^* (q^\prime) = q$. 
	\item Two collections of polygons define the same dilation surface if they are the same up to translation, rotation by $180^\circ$, dilation, and cut and paste moves. 
\end{enumerate}

For example, the three dilation surfaces in Figure \ref{fig:equivalent} are equivalent. 
	
	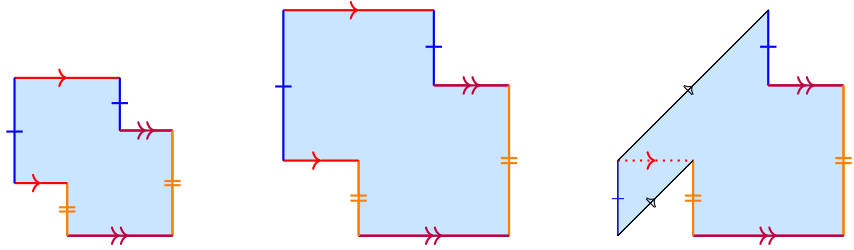
\begin{figure}[ht]
	\begin{center}
		\begin{tikzpicture}[scale = 0.7]
		\draw[fill={rgb,255:red,202; green,229; blue,255}] (1,0) -- (3,0) -- (3,2) -- (2,2) -- (2,3) -- (0,3) -- (0,1) -- (1,1) -- (1,0); 
		
		\draw[red, thick,->] (0,3) -- (1,3);
		\draw[red, thick] (1,3) -- (2,3);
		
		\draw[blue, thick,-|] (2,3) -- (2,2.5);
		\draw[blue, thick] (2,2.5) -- (2,2);
		
		\draw[purple, thick] (2,2) -- (3,2);
		\draw[purple, thick,->] (2,2) -- (2.5,2);
		\draw[purple, thick,>-] (2.5,2) -- (3,2);
		
		\draw[orange, thick] (3,2) -- (3,0);
		\draw[orange, thick,-|] (3,2) -- (3,1.02);
		\draw[orange, thick,|-] (3,0.98) -- (3,0);
		
		\draw[purple, thick] (1,0) -- (3,0);
		\draw[purple, thick,->] (1,0) -- (2,0);
		\draw[purple, thick,>-] (2,0) -- (3,0);

		\draw[orange, thick] (1,0) -- (1,1);
		\draw[orange, thick,-|] (1,0) -- (1,0.48);
		\draw[orange, thick,|-] (1,0.52) -- (1,1);
		
		\draw[red, thick,->] (0,1) -- (0.5,1);
		\draw[red, thick] (0.5,1) -- (1,1);
		
		\draw[blue, thick,-|] (0,1) -- (0,2);
		\draw[blue, thick] (0,2) -- (0,3);
		
		\end{tikzpicture}
		\hspace{1cm}
		\begin{tikzpicture}
		\draw[fill={rgb,255:red,202; green,229; blue,255}] (1,0) -- (3,0) -- (3,2) -- (2,2) -- (2,3) -- (0,3) -- (0,1) -- (1,1) -- (1,0); 
		
		\draw[red, thick,->] (0,3) -- (1,3);
		\draw[red, thick] (1,3) -- (2,3);
		
		\draw[blue, thick,-|] (2,3) -- (2,2.5);
		\draw[blue, thick] (2,2.5) -- (2,2);
		
		\draw[purple, thick] (2,2) -- (3,2);
		\draw[purple, thick,->] (2,2) -- (2.5,2);
		\draw[purple, thick,>-] (2.5,2) -- (3,2);
		
		\draw[orange, thick] (3,2) -- (3,0);
		\draw[orange, thick,-|] (3,2) -- (3,1.02);
		\draw[orange, thick,|-] (3,0.98) -- (3,0);
		
		\draw[purple, thick] (1,0) -- (3,0);
		\draw[purple, thick,->] (1,0) -- (2,0);
		\draw[purple, thick,>-] (2,0) -- (3,0);

		\draw[orange, thick] (1,0) -- (1,1);
		\draw[orange, thick,-|] (1,0) -- (1,0.48);
		\draw[orange, thick,|-] (1,0.52) -- (1,1);
		
		\draw[red, thick,->] (0,1) -- (0.5,1);
		\draw[red, thick] (0.5,1) -- (1,1);
		
		\draw[blue, thick,-|] (0,1) -- (0,2);
		\draw[blue, thick] (0,2) -- (0,3);
		
		\end{tikzpicture}
		\hspace{1cm}
		\begin{tikzpicture}
		\draw[fill={rgb,255:red,202; green,229; blue,255}] (0,0) -- (1,1) -- (1,0) -- (3,0) -- (3,2) -- (2,2) -- (2,3) -- (0,1) -- (0,0); 
		
		\draw[blue, thick,-|] (2,3) -- (2,2.5);
		\draw[blue, thick] (2,2.5) -- (2,2);
		
		\draw[purple, thick] (2,2) -- (3,2);
		\draw[purple, thick,->] (2,2) -- (2.5,2);
		\draw[purple, thick,>-] (2.5,2) -- (3,2);
		
		\draw[orange, thick] (3,2) -- (3,0);
		\draw[orange, thick,-|] (3,2) -- (3,1.02);
		\draw[orange, thick,|-] (3,0.98) -- (3,0);
		
		\draw[purple, thick] (1,0) -- (3,0);
		\draw[purple, thick,->] (1,0) -- (2,0);
		\draw[purple, thick,>-] (2,0) -- (3,0);

		\draw[orange, thick] (1,0) -- (1,1);
		\draw[orange, thick,-|] (1,0) -- (1,0.48);
		\draw[orange, thick,|-] (1,0.52) -- (1,1);
		
		\draw[red, thick,->,dotted] (0,1) -- (0.5,1);
		\draw[red, thick, dotted] (0.5,1) -- (1,1);

		\draw[blue, mid = {|}] (0,0) -- (0,1); 
		\draw[black, mid = {|>}] (0,0) -- (1,1); 
		\draw[black, mid = {|>}] (0,1) -- (2,3); 
		
		\end{tikzpicture}
	\end{center} 
	\caption{Three equivalent dilation surfaces. The surfaces are related to each other by translation, dilation, and cut and paste operations.}
	\label{fig:equivalent}
	\end{figure}

Another construction that will be important in this paper is the \textbf{slit construction}, a connect sum operation of two dilation surfaces to create another dilation surface. In this construction, we start with two dilation surfaces and cut a slit in each surface in the same direction. Then, we connect sum the two surfaces together by gluing one side of each slit to the opposite side of the slit on the other surface. 
	
	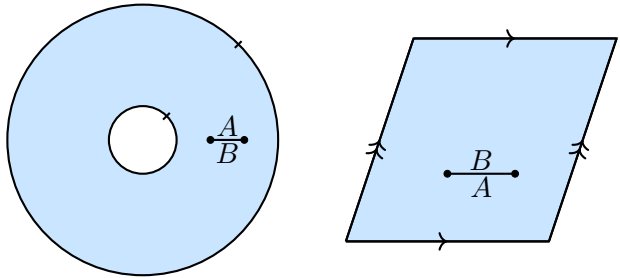
\begin{figure}[ht]
	\begin{center}
		\begin{tikzpicture}[scale = 0.9]
		\draw[thick, fill = {rgb,255:red,202; green,229; blue,255}] (0,0) circle [radius=2];
		\draw[thick, fill = white] (0,0) circle [radius=0.5];
		\draw[thick] (0.3, 0.3) -- (.4, .4); 
		\draw[thick] (1.36, 1.36) -- (1.46, 1.46); 
		
		\draw[thick] (1,0) -- (1.5,0); 
		\draw[fill] (1,0) circle [radius = 0.05]; 
		\draw[fill] (1.5,0) circle [radius = 0.05]; 
		\node at (1.25, 0.2) {$A$}; 
		\node at (1.25, -0.2) {$B$}; 
		
		\draw[thick, fill = {rgb,255:red,202; green,229; blue,255}] (3,-1.5) -- (6,-1.5) -- (7,1.5) -- (4,1.5) -- (3, -1.5); 
		\draw[thick, mid = {>}] (3,-1.5) -- (6,-1.5);
		\draw[thick, mid = {>}] (4,1.5) -- (7,1.5);
		\draw[thick, mid = {>>}] (3,-1.5) -- (4,1.5); 
		\draw[thick, mid = {>>}] (6,-1.5) -- (7,1.5); 
		
		\draw[thick] (4.5,-0.5) -- (5.5,-0.5); 
		\draw[fill] (4.5, -0.5) circle [radius = 0.05]; 
		\draw[fill] (5.5, -0.5) circle [radius = 0.05]; 
		\node at (5, -0.3) {$B$}; 
		\node at (5, -0.7) {$A$}; 
		\end{tikzpicture} 
	\end{center}
	\caption{The slit construction gives us a genus $2$ dilation surface as the connect sum of two genus $1$ dilation surfaces. Sides labeled by the same letter are identified.}
	\label{fig:slit}
	\end{figure}

	In Figure \ref{fig:slit}, slits are cut in two genus one dilation surfaces. After identifications, we get a genus two dilation surfaces with two cone points of angle $4\pi$ each. In general, when each slit is drawn so that it does not intersect any existing cone points on the surface, the slit construction takes two genus $g_1$ and $g_2$ surfaces and produces a genus $g_1+g_2$ surface with two extra cone points of angle $4\pi$ each. 

\subsection{The mapping class group}
\label{subsec:mcg}

In this paper, we will be interested in the interplay between affine automorphisms on dilation surfaces, which are maps that respect the geometry of the surface, and mapping class group elements, which are topological matps. We briefly recall here some key facts about the mapping class group of a surface. For more details, please see \cite{FM}. 

The \textbf{mapping class group} of a topological surface $\Sigma$, often denoted as $\text{Mod}(\Sigma)$, is the group of homeomorphisms of $\Sigma$ up to isotopy. Sometimes, we will restrict our attention only to the orientation-preserving elements in $\text{Mod}(\Sigma)$. An important set of elements in the mapping class group are Dehn twists, in part because they generate $\text{Mod}(\Sigma)$. A \textbf{Dehn twist} is a twist around a simple closed curve on the surface. More specifically, if $T_\alpha$ is a Dehn twist around a curve $\alpha$ on a surface $\Sigma$, then (up to isotopy) there is a tubular neighborhood of $\alpha$ that is homoeomorphic to  $S^1 \times [0,1]$ for which $T_\alpha (\theta, t) = (\theta + 2 \pi t, t)$. Away from this tubular neighborhood, $T_\alpha$ is the identity. \textbf{Dehn multitwists} are generalizations of Dehn twists where now a single map $T$ can consist of multiple Dehn twists or powers of Dehn twists around disjoint simple closed curves. 

Another important fact about the mapping class group of a surface is the Nielson-Thurston classification for mapping class group elements. 

\begin{thm}[Nielson-Thurston classification] If $f$ is a homeomorphism of a closed genus $g \geq 2$ surface, then $f$ is isotopic to a homeomorphism $h$ of the surface such that $h$ is periodic, reducible, or pseudo-Anosov. 
\end{thm}

A map $h : S \rightarrow S$ is \textbf{periodic} if some power of $h$ is the identity map. $h$ is said to be \textbf{reducible} if it preserves a nonempty set of disjoint simple closed curves. Finally, $h$ is said to be \textbf{pseudo-Anosov} if there exists a pair of transverse measured foliations $(\mathcal{F}_u, \mu_u)$ and $(\mathcal{F}_s, \mu_s)$ on the surface and a $\lambda > 1$ for which $$h \cdot (\mathcal{F}_u, \mu_u) = (\mathcal{F}_u, \lambda \mu_u) \text{ and } h \cdot (\mathcal{F}_s, \mu_s) = (\mathcal{F}_s, \lambda^{-1} \mu_s).$$ The measured foliations $(\mathcal{F}_u, \mu_u)$ and $(\mathcal{F}_s, \mu_s)$ are often referred to as the \textbf{unstable foliation} and the \textbf{stable foliation} respectively. The idea of a pseudo-Anosov map is that way from finitely many points, the map locally looks like the map $(x,y) \mapsto (\lambda x, \lambda^{-1} y)$ for some $\lambda > 1$. 

\subsection{Affine automorphisms and Veech groups} 
\label{subsec:aff} 

Let us now examine the geometric side of things. A main object of interest for us will be the Veech groups and affine automorphism groups of dilation surfaces.  

The group of \textbf{affine automorphisms} of a dilation surface $(X, \rho, q)$, denoted by $\Aff(X, \rho, q)$, is the group of orientation-preserving homeomorphisms $f : X \rightarrow X$ that are locally real affine in the coordinates given by the twisted quadratic differential $q$, away from the singularities of $q$. These elements can also be though of as orientation-preserving real affine maps of the universal cover of $(X, \rho, q)$, up to deck transformations. We see that while a lift $\tilde{f}$ has well-defined derivative $D(\tilde{f})$, the derivative $D(f)$ is only well-defined up to scale and so is an element of $\text{PGL}_2^+\R \cong \text{PSL}(2, \mathbb{R})$. We will call the image of $\Aff(X, \rho, q)$ in $\text{PSL}(2, \mathbb{R})$ the \textbf{Veech group} of the surface and denote is $\text{SL}(X, \rho, q)$. We can also think of this group as the stabilizer of the surface under the action of $\text{PSL}(2, \mathbb{R})$.

In this paper, we will be concerned with the question of when a mapping class group element can be represented in $\Aff(X, \rho, q)$ for some dilation surface. This is intimately connected to the projection map $\Aff(X, \rho, q) \rightarrow \text{Mod}(\Sigma)$, where $\text{Mod}(\Sigma)$ denotes the mapping class group of the underlying topological surface $\Sigma$ to the Riemann surface $X$. We also make the following observation. 

\begin{prop} The projection map $\Aff(X, \rho, q) \rightarrow \text{Mod}(\Sigma)$ is not necessarily injective. 
\end{prop} 

\begin{proof} We will demonstrate that the projection map is not necessarily injective by finding a dilation surface for which this is true. Let us consider the surface in Figure \ref{fig:2hopf}.
	
	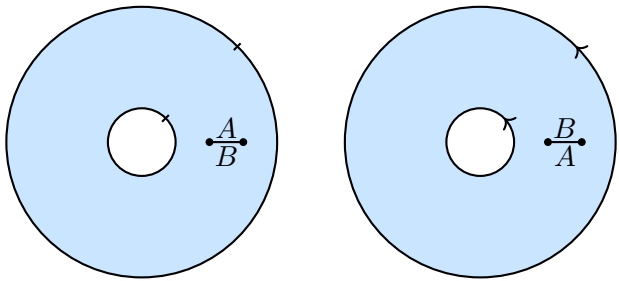
\begin{figure}[ht]
	\begin{center} 
		\begin{tikzpicture}[scale = 0.9]
		\draw[thick, fill = {rgb,255:red,202; green,229; blue,255}] (0,0) circle [radius=2];
		\draw[thick, fill = white] (0,0) circle [radius=0.5];
		\draw[thick] (0.3, 0.3) -- (.4, .4); 
		\draw[thick] (1.36, 1.36) -- (1.46, 1.46); 
		
		\draw[thick, fill = {rgb,255:red,202; green,229; blue,255},
		decoration={markings, mark=at position 0.125 with {\arrow{>}}},
		postaction={decorate} 
		] (5,0) circle [radius = 2];
		\draw[thick, decoration={markings, mark=at position 0.125 with {\arrow{>}}}, postaction={decorate}, fill = white] (5,0) circle [radius=0.5];
		\draw[thick] (1,0) -- (1.5,0); 
		\draw[thick] (6,0) -- (6.5,0); 
		\draw[fill] (1,0) circle [radius = 0.05]; 
		\draw[fill] (1.5,0) circle [radius = 0.05]; 
		\draw[fill] (6,0) circle [radius = 0.05]; 
		\draw[fill] (6.5,0) circle [radius = 0.05]; 
		\node at (1.25, 0.2) {$A$}; 
		\node at (1.25, -0.2) {$B$}; 
		\node at (6.25, 0.2) {$B$}; 
		\node at (6.25, -0.2) {$A$}; 
		\end{tikzpicture} 
	\end{center}
	\caption{The connect sum of two Hopf tori.}
	\label{fig:2hopf}
	\end{figure}

	This surface is the connect sum of two Hopf tori, where a horizontal slit is cut in each torus and the top of each slit is identified with the bottom of the slit on the other torus. Then, we can check that any real affine map that either preserves each slit or permutes the two slits is in $\Aff(X, \rho, q)$. These maps have derivatives of the form $\begin{bmatrix} a & b \\ 0 & 1/a \end{bmatrix}$, and all matrices of this form are the derivative of such a map. The set of such maps forms two connected components, one where the slits are swapped and one where they are preserved. Since the map $\Aff(X, \rho, q) \rightarrow \text{Mod}(\Sigma)$ is continuous, and $\text{Mod}(\Sigma)$ is discrete, this implies that each connected component is mapped to one mapping class group element (specifically, the identity and the element that swaps the two tori). The preimage of either of these elements is then infinite. 
\end{proof} 

The above proof and example shows that not only can the map $\Aff(X, \rho, q)$ be not injective, but the preimage of an element of $\text{Mod}(\Sigma)$ can be quite large. We will keep this in mind as we investigate the representability of mapping class group elements in the affine automorphism group of dilation surfaces. 

\section{The realization problem}
\label{sec:realization}

In this section, we will discuss the realization problem, the main topic of this paper. We will start by carefully stating the problem and giving some background on what is known for translation surfaces. An important tool in tackling the realization problem for translation surfaces is the correspondence between the trace of the derivative of an affine automorphism and the Nielsen-Thurston classification of the underlying surface homeomorphism. We will see that this correspondence no longer holds for twisted quadratic differentials and demonstrate this explicitly by showing the existence of exotic Dehn twists. 

\subsection{Statement of the realization problem} 

There is a well-defined map $$p: \Aff(X, \rho, q) \rightarrow \text{Mod}(\Sigma)$$ taking an affine automorphism to the isotopy class of its underlying surface homeomorphism. One big question that we can ask is then the following: 

\textbf{The Realization Problem:} For which elements or subgroups $H \subset \text{Mod}(\Sigma)$ does there exist a dilation surface $(X, \rho, q)$ for which $p (\Aff(X, \rho, q)) \supset H$? 

In other words, what elements or subgroups of the mapping class group can be realized in the affine automorphism group of some dilation surface? To answer this question, one necessary condition is that for every $g \in H$, $g^* \rho = \rho$. Taking this into account, a more refined question that we can ask is:  

\textbf{Question:} Given a subgroup $H \subset \text{Mod}(\Sigma)$, for which $\rho \in \pi^1(\Sigma, \R_+)^H$, the $H$-invariant holonomy representations, does there exist a dilation surface for which $p(\Aff(X, \rho, q)) \supset H$? 

For translation surfaces, there is a well known correspondence between the topological behavior of an affine automorphism $\varphi$ of a translation surface and whether $D \varphi$ is elliptic, parabolic, or hyperbolic (for example, see \cite{Mo}). We recall that a matrix is elliptic, parabolic, or hyperbolic if the absolute value of its trace is less than $2$, equal to $2$, or greater than $2$ respectively. 

\begin{prop} 
	\label{prop:transcorr}If $\varphi$ is an affine automorphism of a translation surface $(X, q)$, then $\varphi$ is periodic if and only if $D\varphi$ is elliptic, reducible if and only if $D \varphi$ is parabolic, and pseudo-Anosov if and only if $D\varphi$ is hyperbolic.
\end{prop}

This correspondence helps us give a partial answer to the realization problem for translation surfaces. Specifically, it helps us narrow down the types of mapping class group elements that can be realized in the affine automorphism group of a translation surface. 

\begin{prop} 
	
\label{prop:classification_translation}
If $\varphi$ is an affine automorphism of a translation surface $(X, q)$, then $\varphi$ is finite-order, psuedo-Anosov, or there exists a power $\varphi^n$ that is a Dehn multitwist. 
\end{prop}
\begin{proof} If $D\varphi$ is elliptic or hyperbolic, then by Proposition \ref{prop:transcorr}, $\varphi$ is finite order or pseudo-Anosov respectively. If $D \varphi$ is parabolic, then by the same proposition, we know that $\varphi$ is reducible and therefore we can find a maximal finite set of disjoint simple closed curves $\mathcal{C}$ preserved by $\varphi$. Some finite power $\varphi^n$ of $\varphi$ then fixes each curve individually. These curves cut the surface into finitely many pieces, and $\varphi^n$ must be reducible on each of these pieces by the trace considerations of Proposition \ref{prop:transcorr}. Since the original curves were chosen to be maximal, it follows that $\varphi^n$ is the identity on these pieces. On each curve of $\mathcal{C}$ itself, $\varphi^n$ can only be some multiple of a Dehn twist. It follows then that $\varphi^n$ is a Dehn multitwist. 
\end{proof}

We would like to find a similar result for dilation surfaces. We recall that for $\varphi \in \Aff(X,\rho,q)$, $D\varphi$ is really only defined up to scale. So when we refer to $D\varphi$ being hyperbolic, parabolic, or elliptic, we are referring to the nature of the trace of $D\varphi$ after scaling so that $\det(D\varphi) = 1$. 

Even with this technicality out of the way, we will soon see that the strict correspondence in Proposition \ref{prop:transcorr} between the type of the derivative matrix $D\varphi$ for $\varphi \in \Aff(X, \rho, q)$ and the classification of the underlying topological map breaks down for dilation surfaces. This will open the door for the possibility of affine automorphisms of dilation surfaces that display mixed behavior, for example being pseudo-Anosov in one part of the surface and a Dehn twist in another. 

Before discussing the realization problem for dilation surfaces in general, let us begin with the case of genus one surfaces. Genus one dilation surfaces are well-understood and have very different dynamics from higher genus surfaces, so we will treat them separately. 

\subsection{Genus One Dilation Surfaces}
\label{sec:genus1mod}

As we will see, there is a stark difference between the the geometry and dynamics of genus one dilation surfaces and that of dilation surfaces of higher genus. Furthermore, in genus one, we understand what every dilation surface can look like and what its Veech group is, whereas in higher genus, much is still unknown. 

We start with some examples of genus one dilation surfaces. The simplest example of a genus one dilation surface is one that resembles a translation surface. These surfaces are represented by parallelograms with parallel sides identified, and have trivial holonomy.

Another example of a genus one dilation surface is an annulus with inner and outer boundaries identified, and example of a Hopf torus. This dilation surface can also be thought of as the exponentiation of a parallelogram torus such that the identifications after exponentiation are via dilation by a real factor and translation. Such a surface is shown in Figure \ref{fig:genus1ex}.
	
	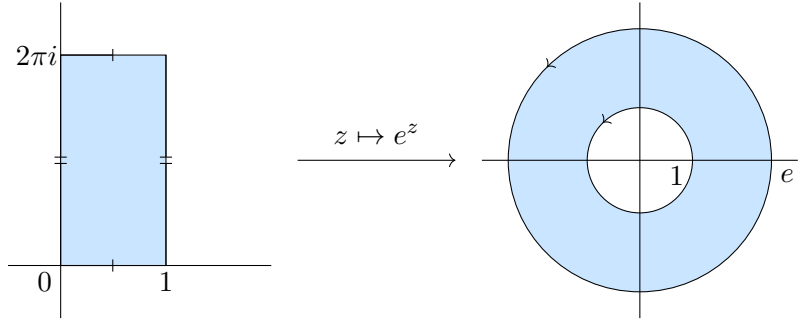
\begin{figure}[ht]
	\begin{center}
		\begin{tikzpicture}[scale = 0.7] 
			\draw[fill = {rgb,255:red,202; green,229; blue,255}, decoration={markings, mark=at position 0.375 with {\arrow{>}}},
			postaction={decorate} ] (9,0) circle [radius=2.5];
			\draw[fill = white, decoration={markings, mark=at position 0.375 with {\arrow{>}}},
			postaction={decorate} ] (9,0) circle [radius=1];
			\draw (-3,-2) -- (2,-2);
			\draw (-2,-3) -- (-2,3); 
			\draw (6,0) -- (12,0);
			\draw (9, -3) -- (9,3); 
			\draw[fill = {rgb,255:red,202; green,229; blue,255}] (-2,-2) -- (0,-2) -- (0,2) -- (-2,2) -- (-2,-2); 
			\draw[-|] (-2,-2) -- (-1,-2); 
			\draw[-|] (-2,2) -- (-1,2); 
			\draw[-|] (-2,-2) -- (-2,-.05);
			\draw[|-] (-2,.05) -- (-2,2);
			\draw[-|] (0,-2) -- (0,-.05);
			\draw[|-] (0,.05) -- (0,2);
			\draw[->] (2.5,0) -- (5.5,0);
			\node at (4,0.5) {$z \mapsto e^z$};
			\node at (0,-2.3) {$1$};
			\node at (-2.3, -2.3) {$0$};
			\node at (-2.45, 2) {$2\pi i$};
			\node at (9.7, -0.3) {$1$}; 
			\node at (11.8, -0.3) {$e$}; 
		\end{tikzpicture} 
	\end{center} 
	\caption{A genus one dilation surface coming from the exponentiation of a translation surface.}
	\label{fig:genus1ex}
	\end{figure}
	
It turns out that all genus one dilation surfaces resemble one of these two examples. Those with trivial holonomy are just translation surfaces, whereas those with nontrivial holonomy are all exponentiations $z \mapsto e^{\alpha z}$ of translation tori. 

We also understand the Veech groups of both of these types of surfaces. In the former translation surface case, the Veech group is isomorphic to $\text{SL}(2, \mathbb{Z})$ and every mapping class group element of the torus is represented in the Veech group. In the latter exponentiated translation surface case, the Veech group is all of $\text{SL}(2,\mathbb{R})$ and therefore only the identity element of the mapping class group is represented in the Veech group. For more details on this latter case, see \cite{DFG}.

\subsection{Exotic Dehn twists for twisted quadratic differentials}
\label{sec:dehn}

Moving on from genus one dilation surfaces, we will now encounter one of the primary difficulties to resolving the realization problem for higher genus dilation surfaces: the existence of exotic Dehn twists. 

As we saw in Proposition \ref{prop:transcorr}, if $\varphi$ is an affine automorphism of a translation surface $(X,q)$ that is a Dehn twist or a Dehn multitwist, then $D\varphi$ is parabolic. The translation surface then decomposes into finitely many Euclidean cylinders and $\varphi$ acts by twisting finitely many times around the core curves of each cylinder. 
	
We call such maps \textbf{standard Dehn multitwists}. Standard Dehn multitwists also exist as affine automorphisms of certain dilation surfaces. But Dehn multitwists on dilation surfaces can also be more exotic: if a Dehn multitwist $\varphi$ twists along the core curves of affine cylinders, curves that now have nontrivial holonomy, then we call it an \textbf{exotic Dehn multitwist}. We will demonstrate the existence of such exotic Dehn multitwists in the following theorem. 

\dehn

\begin{proof}
	
	We will prove this theorem by constructing explicit examples of dilation surfaces with an exotic Dehn multitwist in their affine automorphism group. To constructe genus $g \geq 3$ examples, we start with a genus $1$ surface that is an annulus in the plane, centered at the origin, with outer and inner rings at radius $2$ and $1$ identified. Alternatively, we think of this surface as $\C^*/z \sim 2z$. 
	
	We recall from our discussion of the genus $1$ case that all of $\text{PSL}(2, \mathbb{R})$ stabilizes this surface. Let us consider the specific matrix $M = \begin{bmatrix} 2 & 0 \\ 0 & 1 \end{bmatrix} \cong \begin{bmatrix} \sqrt{2} & 0 \\ 0 & 1/\sqrt{2}\end{bmatrix} \in \text{PSL}(2, \mathbb{R})$. Applying this matrix to our surface corresponds to a map that is preserves the horizontal and the vertical foliations and stretches the whole surface horizontally. Figure \ref{fig:genus1stretch} shows this surface as well as four important curves $a_1, a_1, b_1,$ and $b_2$ that are fixed pointwise by $M$. 
	
	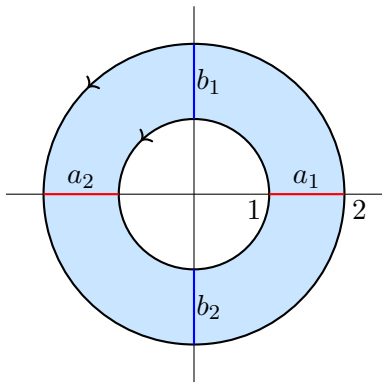
\begin{figure}[ht]
	\begin{center}
		\begin{tikzpicture}
			\draw[thick, fill = {rgb,255:red,202; green,229; blue,255}, decoration={markings, mark=at position 0.375 with {\arrow{>}}},
			postaction={decorate} ] (0,0) circle [radius=2];
			\draw[thick, fill = white, decoration={markings, mark=at position 0.375 with {\arrow{>}}},
			postaction={decorate} ] (0,0) circle [radius=1];
			\draw (-2.5, 0) -- (2.5, 0);
			\draw (0, -2.5) -- (0, 2.5); 
			\node at (0.8, -0.2) {$1$}; 
			\node at (2.2, -0.2) {$2$}; 
			
			\draw[thick, red,] (1,0) -- (2,0); 
			\draw[thick, red] (-1,0) -- (-2,0); 
			\draw[thick, blue] (0,1) -- (0,2); 
			\draw[thick, blue] (0,-1) -- (0,-2); 
			\node at (1.5, 0.2) {$a_1$}; 
			\node at (-1.5, 0.2) {$a_2$}; 
			\node at (0.2,1.5) {$b_1$}; 
			\node at (0.2,-1.5) {$b_2$}; 
		\end{tikzpicture} 
	\end{center}
	\caption{A genus $1$ dilation surface. The curves marked $\alpha_1, \alpha_2, \beta_1,$ and $\beta_2$ are fixed pointwise by the map $M = \begin{bmatrix} 2 & 0 \\ 0 & 1 \end{bmatrix}$.}
	\label{fig:genus1stretch}
	\end{figure} 
	
	The map $M$ will fix pointwise the horizontal curves $a_1$ and $a_2$, as well as the vertical curves $b_1$ and $b_2$. Topologically, the map $M$ is the product of two point pushing maps one time around the curves $a_1$ and $a_2$ in the radially outward direction. This map is isotopic to the identity.

	One way that we can create higher genus examples of exotic Dehn twists is by building off of this map and adding slits to the surface along the fixed curves $\alpha_i$ and $\beta_j$ via the slit or connect sum construction. Figure \ref{fig:exoticDehn} gives an example of a genus $3$ surface with an exotic Dehn multitwist in its affine automorphism, given by the matrix $M$. 
	
	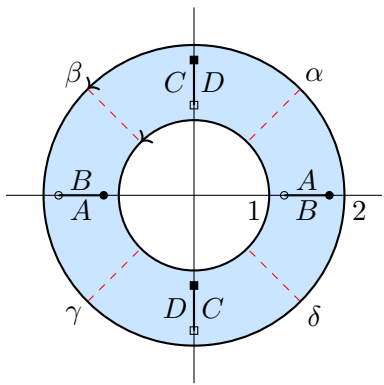
\begin{figure}[ht]
	\begin{center}
		\begin{tikzpicture}
			\draw[thick, fill = {rgb,255:red,202; green,229; blue,255}, decoration={markings, mark=at position 0.375 with {\arrow{>}}},
			postaction={decorate} ] (0,0) circle [radius=2];
			\draw[thick, fill = white, decoration={markings, mark=at position 0.375 with {\arrow{>}}},
			postaction={decorate} ] (0,0) circle [radius=1];
			\draw (-2.5, 0) -- (2.5, 0);
			\draw (0, -2.5) -- (0, 2.5); 
			\node at (0.8, -0.2) {$1$}; 
			\node at (2.2, -0.2) {$2$}; 
			
			\draw[thick] (1.2,0) -- (1.8,0); 
			\draw  (1.2,0) circle [radius = 0.05]; 
			\draw[fill] (1.8,0) circle [radius = 0.05]; 
			
			\draw[thick] (-1.2,0) -- (-1.8,0); 
			\draw[fill]  (-1.2,0) circle [radius = 0.05]; 
			\draw (-1.8,0) circle [radius = 0.05]; 
			
			\draw[thick] (0,1.2) -- (0,1.8);
			\draw (-0.05, 1.15) rectangle (0.05, 1.25); 
			\draw[fill] (-0.05, 1.75) rectangle (0.05, 1.85); 
			
			\draw[thick] (0,-1.2) -- (0,-1.8);
			\draw[fill] (-0.05, -1.15) rectangle (0.05, -1.25); 
			\draw (-0.05, -1.75) rectangle (0.05, -1.85); 
			
			\node at (1.5, 0.2) {$A$}; 
			\node at (1.5, -0.2) {$B$}; 
			\node at (-1.5, 0.2) {$B$}; 
			\node at (-1.5, -0.2) {$A$}; 
			
			\node at (-0.25, 1.5) {$C$}; 
			\node at (0.25, 1.5) {$D$}; 
			\node at (0.25, -1.5) {$C$}; 
			\node at (-0.25, -1.5) {$D$}; 
			
			\draw[red,dashed] (1.41, 1.41) -- (0.705, 0.705); 
			\draw[red,dashed] (1.41, -1.41) -- (0.705, -0.705); 
			\draw[red,dashed] (-1.41, 1.41) -- (-0.705, 0.705); 
			\draw[red,dashed] (-1.41, -1.41) -- (-0.705, -0.705); 
			\node at (1.6,1.6) {$\alpha$};
			\node at (-1.6,1.6) {$\beta$};
			\node at (-1.6,-1.6) {$\gamma$};
			\node at (1.6,-1.6) {$\delta$};
		\end{tikzpicture} 
	\end{center} 
	\caption{A genus $3$ surface for which the map $M = \begin{bmatrix} 2 & 0 \\ 0 & 1 \end{bmatrix}$ is an exotic Dehn multitwist on the surface.}
	\label{fig:exoticDehn}
	\end{figure}
	
	In this surface, we took our genus $1$ example and added a pair of horizontal slits and a pair of vertical slits. Opposite sides of paired slits are identified, and each pair of slits adds one to the total genus of the surface. Since the slits are placed on the horizontal and vertical curves of the surface, which are fixed pointwise by $M$, $M$ preserves the slits. The map $M$ is then the product of Dehn twists $T_\alpha^{-1}$, $T_\beta$, $T_\gamma^{-1}$, and $T_\delta$. These four Dehn twists are all through curves that have nontrivial holonomy. The presence of the slits also stops this map from being isotopic to the identity. To obtain a genus $g \geq 3$ example, we add $g-1$ pairs of slits to the horizontal and vertical curves $\alpha_i$ and $\beta_j$ of Figure \ref{fig:genus1stretch}, with at least one slit on each of the four curves so that the resulting map is not isotopic to the identity. 
\end{proof} 

This example shows the existence of a new phenomenon for dilation surfaces that does not occur for translation surfaces: a reducible mapping class group element that is represented by a hyperbolic matrix. More specifically, we have a Dehn multitwist around loops that are not the core curves of Euclidean cylinders, but around the core curves of affine cylinders. We also notice that in our examples, the derivative of our exotic Dehn twists as an element of $\text{PSL}(2, \mathbb{R})$ is always hyperbolic, breaking the strict correspondence between the trace of the Veech group matrix and the classification of the topological map as given in Proposition \ref{prop:transcorr}. 


\subsection{Trace of affine automorphism group elements} 

For translation surfaces, the trace of an affine automorphism element dictates what type of mapping class group element it is. The underlying surface homeomorphism is periodic, reducible, or pseudo-Anosov corresponding to whether the derivative of the affine automorphism is an elliptic, parabolic, or hyperbolic matrix. 

We have already seen that the existence of exotic Dehn twists breaks this correspondence for dilation surfaces. That is, there exist affine automorphisms of dilation surfaces that are Dehn multitwists (and therefore reducible), but have a hyperbolic derivative. We will show that this is the only exception to the correspondence between the derivative matrix type and Nielson-Thurston classification. This will then allow us to give a classification similar to Proposition \ref{prop:classification_translation} but now for dilation surfaces. We recall that a separatrix is a straight line trajectory emanating from a cone point and a saddle connection is a straight line trajectory connecting two cone points. We begin with the following lemma. 

\begin{lem} 
	\label{lem: parabolic}
	If $(X, \rho, q)$ is a dilation surface and $\varphi \in \Aff(X, \rho, q)$ with $D\varphi$ parabolic, then every separatrix of $(X, \rho, q)$ in the shearing direction of $D\varphi$ is a saddle connection. 
\end{lem} 

\begin{proof} We may assume without loss of generality that $\varphi$ shears in the horizontal direction. Then $\varphi$ must permute the cone points and the horizontal separatrices of $(X, \rho, q)$. First, can find a power $\varphi^n$ of $\varphi$ such that $\varphi^n$ fixes every cone point and horizontal separatrix on the surface. If we pick any singularity $p$ of the surface, it has a neighborhood that is homeomorphic to a neighborhood of the origin in $\R_k^2$, the $k$-fold cyclic cover of $\R^2$ fully ramified at the origin. 
	
	We can then continue radially developing this chart around $p$. In any radial direction from the singularity, we can develop this chart until we hit a singularity. In this way, the saddle connections emanating from our original singularity $p$ can be thought of as a set of points in $\R_k^2$. It is known that this set is always discrete (see Proposition 7 in \cite{DFG}). This set is also well-defined up to scaling by $\R$, with the scaling determined by our original choice of chart around the singularity. 
	
	Then, since we assumed that $\varphi$ was locally a horizontal shear, if we keep this choice of chart around $p$ consistent, $$D\varphi^n = \begin{bmatrix} a & b \\ 0 & a \end{bmatrix}$$ for some $a \in \R_+, b \in \R$. $D \varphi^n$ and $D\varphi^{-n}$ both act on and preserve the saddle connection set as a subset of $\R_k^2$. Since there is a neighborhood of $p$ without any other singularities, we must have that $a = 1$ in the matrix for $D\varphi^n$ in a neighborhood of $p$. Hence, $\varphi^n$ is the identity on the horizontal separatrices coming out of $p$. Since $p$ was arbitrary, $\varphi^n$ is the identity on every horizontal separatrix of $(X, \rho, q)$. 
	
	Now, if $(X, \rho, q)$ had an infinite horizontal separatrix, then by compactness of the surface, there would be a sequence of distinct points $q_i$ on the separatrix that converge to some point $q$ on the surface. Since $\varphi^n$ fixes each $q_i$ and $\varphi^n$ is linear in a neighborhood of $q$, it follows that $D\varphi^n = I$ and $\varphi^n$ is the identity map on this neighborhood of $q$ and therefore the whole surface. 
	
	Hence, if $\varphi^n$ is nontrivial, every horizontal sepatrix must be finite and therefore a saddle connection. 
\end{proof} 

With this lemma, we can now prove the following proposition that states that with the exception of some reducible $\varphi \in \Aff(X, \rho, q)$ with hyperbolic derivative $D \varphi$, the correspondence between type of matrix element of the derivative and type of mapping class group element of the surface homeomorphism still holds. 

\begin{prop} 
	\label{prop:trace}
	Let $\varphi$ be an affine automorphism of the dilation surface $(X, \rho, q)$, and let $D\varphi \in \text{SL}(2,\mathbb{R})$. If $D\varphi$ is elliptic, then $\varphi$ is finite order. If $D\varphi$ is parabolic, then $\varphi$ is reducible. If $D\varphi$ is hyperbolic, then $\varphi$ is either pseudo-Anosov or reducible. 
\end{prop}

\begin{proof}
	Let $\varphi$ be an affine automorphism of the dilation surface $(X,\rho, q)$, and let $D\varphi$ be the derivative of $\varphi$, thought of as an element of $\text{PSL}(2, \mathbb{R})$. If $D\varphi$ is elliptic, then it is conjugate to a rotation matrix. By Proposition 8 in \cite{DFG}, we know that $\text{SL}(X, \rho, q)$ is discrete if and only if there exist saddle connections in more than one direction on the surface. If $\text{SL}(X, \rho, q)$ is discrete, then $D\varphi$ must be conjugate to a finite order rotation. If $\text{SL}(X, \rho, q)$ is not discrete, then the surface has saddle connections in only one direction, and so $D\varphi$ is conjugate either to the identity or rotation by $180^\circ$ matrix, because the automorphism must preserve the set of directions of the saddle connections. In either case, $\varphi$ must then be a finite order automorphism of the surface. 
	
	If $D\varphi$ is parabolic, then by Lemma \ref{lem: parabolic}, every separatrix in the shearing direction is a finite saddle connection. Some power $\varphi^n$ then must fix every finite saddle connection in that direction and therefore the curves on the surface that are the union of these saddle connections, implying that $\varphi^n$ and therefore $\varphi$ is reducible. 
	
	If $D\varphi$ is hyperbolic, then $\varphi$ cannot be finite order, since this would imply that $D\varphi$ is finite order. Therefore, $\varphi$ must be reducible or pseudo-Anosov. 
\end{proof}

We have already seen that if $D\varphi$ is hyperbolic, it can be an exotic Dehn twist, which is a reducible map. We will see later on that it can also be pseudo-Anosov. 

We wish to understand exactly what types of topological maps can arise as affine automorphisms of dilation surfaces. Proposition \ref{prop:trace} resolves this question for $D\varphi$ elliptic. The next proposition will show that just like for translation surfaces, $D\varphi$ being parabolic will imply that some power $\varphi^n$ of $\varphi$ is a Dehn multitwist.

\begin{prop}
	\label{prop:parabolic}
	If $\varphi$ is in $\Aff(X, \rho, q)$ and $D\varphi$ is parabolic, then some power of $\varphi$ is a standard Dehn multitwist. 
\end{prop}
\begin{proof} Let $\varphi \in \Aff(X, \rho, q)$ be such that $D \varphi$ is parabolic. Without loss of generality, we may assume that the horizontal direction is the direction of shearing. Then, by Lemma \ref{lem: parabolic}, we know that every horizontal separatrix of $(X, \rho, q)$ is a saddle connection. If we cut the surface along the horizontal saddle connections, the remaining pieces are foliated by the horizontal foliation and have no singularites. This implies by Euler characteristic considerations that each piece of the cut-apart surface is a foliated annulus.
	
	We saw in the proof of Proposition \ref{prop:parabolic} that some power $\varphi^n$ of $\varphi$ must fix every horizontal saddle connection and therefore we have that $$D \varphi^n = \begin{bmatrix} 1 & c \\ 0 & 1 \end{bmatrix}$$ in the neighborhood of every horizontal saddle connection.

	For each annuluar region, if the holonomy around the central curve is trivial, then the annulus must be a Euclidean annulus, as shown in Figure \ref{fig:euc_annulus}. The parabolic element $\varphi^n$ will then be a Dehn twist around the central curve of the cylinder. 
	
	\begin{figure}[ht]
	\begin{center}
		\begin{tikzpicture}
			\draw[thick] (0,0) -- (4,0); 
			\draw[thick] (0,2) -- (4,2); 
			\draw[thick, mid = {>}] (0,0) -- (0,2); 
			\draw[thick, mid = {>}] (4,0) -- (4,2); 
			\draw[pattern=horizontal lines,draw=none, pattern color = {rgb,255:red,0; green,128; blue,128}] (0,0) rectangle (4,2); 
		\end{tikzpicture} 
	\end{center}
	\caption{A Euclidean annulus with its horizontal foliation.} 
	\label{fig:euc_annulus}
	\end{figure}
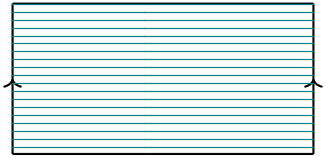
	
	If the holonomy around the central curve is nontrivial, then it must be foliated with a Reeb foliation. The horizontal boundary curves must then be closed curves that the Reeb foliation accumulates on. The annuluar region is then a dilation cylinder of angle $\pi n$, and the associated foliation is $n$ copies of the Reeb foliation glued together along their boundary curves. Figure \ref{fig:dil_annulus} shows a Reeb-foliated annulus and the corresponding dilation cylinder.

	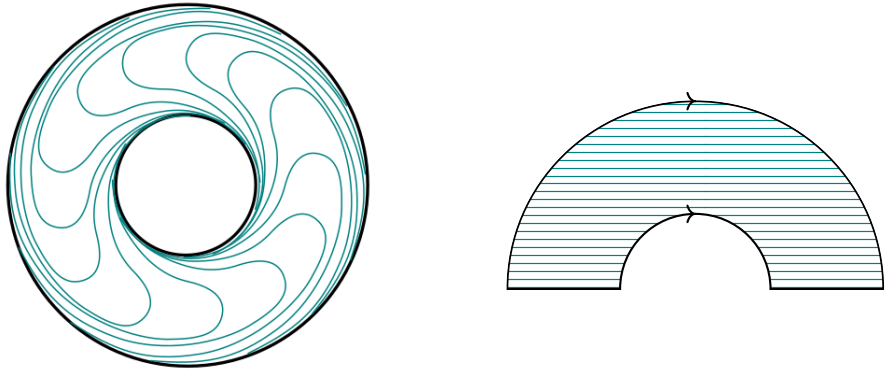
\begin{figure}[ht]
	\begin{center}
		\includegraphics[width=50mm]{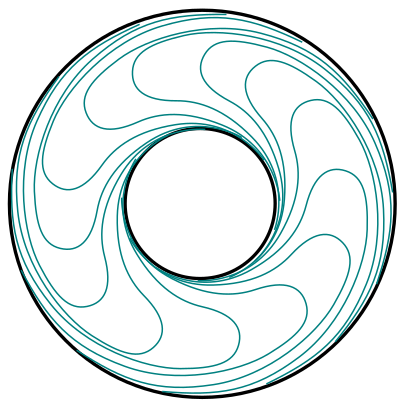}
		\hspace{1cm}
		\begin{tikzpicture}[baseline=-7ex]
			\scope 
			\clip (-3, 0) rectangle (3,3);
			\scope
			\clip (0,0) circle (2.51);
			\draw[pattern=horizontal lines,draw=none, pattern color = {rgb,255:red,0; green,128; blue,128}] (0,0) circle (2.5);
			
			\draw[thick] (0,0) circle [radius = 2.5];
			\draw[thick,fill=white] (0,0) circle [radius = 1];
			\draw[thick] (-2.5,0) -- (-1,0);
			\draw[thick] (2.5,0) -- (1,0);
			\endscope
			\endscope 
			\draw[thick, mid = {>}] (-0.01,1) -- (0.01,1); 
			\draw[thick, mid = {>}] (-0.01,2.5) -- (0.01,2.5); 
		\end{tikzpicture}
		\end{center}
		\caption{On the left, we have an annulus with a Reeb foliation. On the right, we see a dilation cylinder of angle $\pi$ whose horizontal foliation gives rise to a Reeb foliation. }
		\label{fig:dil_annulus}
		\end{figure}
	
	Then, we see that $\varphi^n$ must act on dilation cylinders by fixing the boundary horizontal boundary curves and shearing the rest of the cylinder by a map that is isotopic to the identity.  
		
	Hence, $\varphi^n$ fixes the horizontal saddle connections, is a standard Dehn twist on the Euclidean cylinders in the complement of these saddle connections, and is isotopic to the identity on the cylinders with holonomy. This implies that $\varphi^n$ is a Dehn multitwist. 
\end{proof}

Finally, we wish to understand what topological maps $\varphi \in \Aff(X, q, \rho)$ could be if $D\varphi$ is hyperbolic. We will achieve this through a series of propositions. The first one below will consider what the surface $(X,q,\rho)$ must look like if $D \varphi$ is hyperbolic and $\varphi$ fixes a saddle connection. 

\begin{prop} 
	\label{prop:hyperbolic} If $\varphi$ is in $\Aff(X, \rho, q)$ such that $D \varphi$ is hyperbolic and $\varphi$ fixes a saddle connection, then $(X, \rho, q)$ can be decomposed into finitely many dilation cylinders of angle $\frac{\pi n}{2}$ such that $\varphi$ is a Dehn multitwist along the core curves of these dilation cylinders. 
\end{prop}

\begin{proof} Up to conjugacy and scaling, we may assume that $D \varphi$ is a positive diagonal matrix and the fixed saddle connection is horizontal. Then, if we look at a chart in a neighborhood of the fixed saddle connection, we have that in those coordinates, $$D \varphi = \begin{bmatrix} 1 & 0 \\ 0 & \lambda\end{bmatrix}$$ with $\lambda > 0$. Let $p$ and $q$ be the (not necessarily distinct) endpoints of the fixed saddle connection. This then tells us that $\varphi$ must also fix pointwise every horizontal separatrix emanating from $p$ or $q$. If any of these separatrices were infinite, then the separatrix would have an accumulation point $s$ that must also be fixed by $\varphi$. In any small neighborhood $U$ of this point $s$, $\varphi$ would be locally linear and fix pointwise this point and any segment of the separatrix passing through $U$, and therefore $D \varphi = I$, the identity matrix. This would then mean that $\varphi$ is the identity map locally and therefore everywhere on the surface. 
	
	Thus we have shown that every horizontal separatrix emanating from either $p$ or $q$, the endpoints of the fixed saddle connection, must be a saddle connection and fixed pointwise by $\varphi$. We can then apply this argument inductively to show that in the graph with vertices the singularities and edges the horizontal separatrices of $(X, \rho, q)$, the connected component containing the original fixed saddle connection consists entirely of saddle connections fixed by $\varphi$. 
	
	Now, let us consider a ribbon graph coming from a neighborhood of this connected component of the graph of horizontal saddle connections. Each boundary component of the ribbon graph is then homotopic to a cycle $C$ of horizontal saddle connections. In a one-sided neighborhood of each cycle $C$, we have that $\lambda > 0$ in the matrix $D\varphi$ as above. If $\lambda = 1$, then $D \varphi = I$ and $\varphi$ would be the identity map. Otherwise when $\lambda \neq 1$, $C$ is attracting or repelling for the horizontal foliation and there must be a nearby non-horizontal periodic trajectory $\gamma$ homotopic to $C$ that is the core curve of a dilation cylinder, one of whose boundary components is $C$. The other boundary component of the cylinder must be either a cycle of vertical or horizontal saddle connections, since these are the only directions fixed by $\varphi$. We show a schematic of this is Figure \ref{fig:ribbon_graph}.
	
	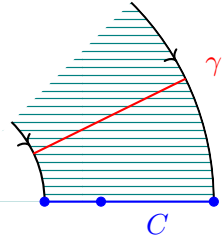
\begin{figure}[ht]
	\begin{center}
		\vspace{-2cm}
		\begin{tikzpicture}[scale=1.5]
			\scope 
			\clip (0, 0) -- (3,0) -- (3,3) -- (0,0);
			\scope
			\clip (0,0) circle (2.51);
			\draw[pattern=horizontal lines,draw=none, pattern color = {rgb,255:red,0; green,128; blue,128}] (0,0) circle (2.5);
			
			\draw[thick] (0,0) circle [radius = 2.5];
			\draw[thick,fill=white] (0,0) circle [radius = 1];
			\endscope
			\endscope 
			\draw[thick, mid = {>}] (.86,.505) -- (.87,.495); 
			\draw[thick, mid = {>}] (2.16,1.255) -- (2.17,1.245); 
			\draw[thick,blue] (2.5,0) -- (1,0);
			\draw[fill, blue] (1,0) circle [radius = 0.04];
			\draw[fill,blue] (1.5,0) circle [radius = 0.04];
			\draw[fill,blue] (2.5,0) circle [radius = 0.04];
			\node[blue] at (2,-.2) {$C$};
			\draw[thick, red] (.9, .425) -- (2.25,1.09);
			\node[red] at (2.5, 1.2) {$\gamma$};
		\end{tikzpicture}
	\end{center}
	\caption{A cycle $C$ of the graph of horizontal saddle connections that is homotopic to a boundary component $\delta$ of the ribbon graph, and a curve $\gamma$ that is a core curve of the dilation cylinder on that portion of the ribbon graph. }
	\label{fig:ribbon_graph}
	\end{figure}

	We can then repeat this process with any fixed horizontal or vertical saddle connection that we have found so far to eventually find a closed subsurface made up of dilation cylinders of angle $\frac{\pi n}{2}$ contained in $(X, \rho, q)$ Since $(X, \rho, q)$ is connected, this closed subsurface must be the whole surface and therefore $(X, \rho, q)$ decomposes into dilation cylinders of angle $\frac{\pi n}{2}$. Since $\varphi$ fixes each boundary component of each of these dilation cylinders, it must act by some multiple of a Dehn twist along the core curve of each of these cylinders.
\end{proof} 

Building off of this proposition, we can now determine more precisely what the topology of a reducible map $\varphi \in \Aff(X,\rho,q)$ with $D \varphi$ hyperbolic can be be. 

\begin{prop} 
	\label{prop:hyperbolic2}
	If $\varphi$ is in $\Aff(X, \rho, q)$ such that $D \varphi$ is hyperbolic and $\varphi$ is reducible, then some power $\varphi^n$ of $\varphi$ is an exotic Dehn multitwist. 
\end{prop}
\begin{proof} Suppose that $\varphi$ is an affine automorphism of $(X, \rho, q)$, that $D \varphi$ is hyperbolic, and that $\varphi$ is reducible. Then, $\varphi$ fixes a set $\mathcal{C} = \{C_1, C_2, \ldots, C_k\}$ of topological curves on the surface and some power $\varphi^m$ will send each curve $C_i$ to itself topologically and preserve its orientation. 
	
	Let us focus our attention on the curve $C_1$. On the surface $(X, \rho, q)$, $C_1$ is either the core curve of a dilation cylinder or a Euclidean cylinder, or has a unique geodesic representative as a collection of saddle connections. In the former case where $C_1$ is the core curve of a cylinder, $\varphi^m$ must then send the cylinder to itself and therefore fixes each boundary component of the cylinder. Since the boundary components of cylinders are made up of saddle connections, some power of $\varphi^m$ must then fix a saddle connection. In the latter case when $C_1$ has a unique geodesic representative as a collection of saddle connections, $\varphi^m$ must take this set of saddle connections to itself and therefore some power of $\varphi^m$ must fix a saddle connection. 
	
	Thus, we have shown that some power $\varphi^n$ of $\varphi$ fixes a saddle connection. Then, by Proposition \ref{prop:hyperbolic}, we have that $\varphi^n$ is a Dehn multitwist along the core curves of a collection of dilation cylinders. 
\end{proof}

Finally, we have the following classification result that shows what the underlying topological maps of affine automorphisms of dilation surfaces can look like. 

\classification 

\begin{proof} This theorem follows directly from Proposition \ref{prop:trace}, Proposition \ref{prop:parabolic}, and Proposition \ref{prop:hyperbolic2}.
\end{proof}

At first glance, comparing the classification results for translation surfaces in Proposition \ref{prop:classification_translation} and dilation surfaces in Theorem \ref{thm:classification_dilation}, it would look like similar mapping class group elements can be represented in the affine automorphism groups of translation surfaces and dilation surfaces are very similar. 

A question that one could ask is whether the set of mapping class group elements realizable by affine automorphisms of translation surfaces is exactly equal to the set of mapping class group elements realizable by affine automorphisms of dilation surfaces with non-trivial holonomy. The answer to this question is no. We notice that any Dehn multitwist on a translation surface must be comprised of only left twists or only right twists, whereas a Dehn multitwist on a dilation surface arising from a hyperbolic element can consist of both left twists and right twists. 

%
%
%

\section{Realizing Elements of the Mapping Class Group}
\label{sec:elements}

While Theorem \ref{thm:classification_dilation} shows that every affine automorphism of a dilation surface is finite order, pseudo-Anosov, or has some power that that is a Dehn multitwist, it doesn't say that every such topological map is realizable in the $\Aff(X, \rho, q)$  for some dilation surface. In this section, we'll start by investigating this question: 

\textbf{Question:} Given a topological map $\psi$ of a genus $g \geq 2$ surface, when can $\psi$ be realized as an affine automorphism $\varphi$ of some dilation surface $(X, \rho, q)$?

We first note that there is necessary condition for $\psi$ to be the underlying topological map to some $\varphi \in \Aff(X, \rho, q)$. We need that $\psi^* \rho = \rho$.  That is, we need that $\psi$ preserves the holonomy $\rho$. Therefore, a more refined question that we could ask is the following: 

\textbf{Question:} Given a topological map $\psi$ of a genus $g \geq 2$ surface and a compatible holonomy representation $\rho$ such that $\psi^* \rho = \rho$, when can $\psi$ be realized as an affine automorphism $\varphi$ of some dilation surface $(X, \rho, q)$?

In this section, we will show that if $\psi$ is finite order or pseudo-Anosov and $\rho$ is a compatible holonomy representation, then $\psi$ is realized in the affine automorphism group of some dilation surface with holonomy $\rho$. 

\subsection{Realizing finite order elements}

Let us start by considering finite order topological maps $\psi$ of genus $g \geq 2$ surfaces. Then, we have the following. 

\begin{prop} Let $\psi$ be a finite order mapping class group element and $\rho : \pi_1(\Sigma) \rightarrow \mathbb{R}_+$ be a compatible holonomy so that $\psi^*\rho = \rho$. Then, there exists a dilation surface with holonomy $\rho$ and with $\psi$ the underlying topological map of some $\varphi$ in $\Aff(X, \rho, q)$. 
\end{prop}

\begin{proof} If $\psi$ is a finite order mapping class group element, then by the Nielsen realization theorem, $\psi$ can be realized as some finite order automorphism $\varphi$ of a Riemann surface $X$. This automorphism $\varphi$ must also satisfy that $\varphi*\rho = r$. Therefore, $\varphi$ induces a linear map $\varphi* : K_X^2 \otimes L_\rho^2 \rightarrow K_X^2 \otimes L_\rho^2$ that takes meromorphic sections with poles of order at most one to meromorphic sections with poles of order at most one. 
	
	We can then find an eigensection $q$ of $K_X^2 \otimes L_\rho^2$ such that $\varphi^* q = \lambda q$ for some root of unity $\lambda$. It follows then that $\varphi$ induces a finite order affine automorphism on the dilation surface $(X, \rho, q)$, and thus $\psi$ is realized in $\Aff(X, \rho, q)$ as the map $\varphi$. 
	
\end{proof}

Now that we have established the realizability of finite order mapping class group elements in the affine automorphism groups of dilation surfaces, let us shift our attention to another important class of maps: pseudo-Anosov elements. 

\subsection{Realizing pseudo-Anosov elements}
\label{subsec:pA}

Pseudo-Anosov maps are some of the more dynamically interesting elements in the mapping class group, but it can be difficult to actually come up with examples of pseudo-Anosov maps. One way to do so, however, is to realize pseudo-Anosov maps in the affine automorphism groups of translation surfaces. It is well-known that any pseudo-Anosov map $\psi$ can be realized as $\varphi \in \Aff(X,q)$ for some half-translation surface $(X,q)$.  To show this, one uses a standard construction of creating a quadratic differential on a Riemann surface out of two filling measured foliations (see \cite{Wr2}, Theorem 2.12 for a reference), 
$$\mathcal{MF} \times \mathcal{MF}\backslash \Delta \rightarrow \mathcal{QD}.$$ 

Furthermore, if $\varphi \in \Aff(X,q)$ has that $|\text{tr}(D\varphi)| > 2$, then $\varphi$ is topologically a pseudo-Anosov map and the eigendirections of $D\varphi$ give the directions of the stable and unstable foliations of $\psi$ on the surface. 

We can also ask if pseudo-Anosov maps can be realized as affine automorphisms of dilation surfaces and what the significance of that would be. We first note that if $\varphi \in \Aff(X,\rho,q)$ is a pseudo-Anosov map, then we established in Theorem \ref{thm:classification_dilation} that  $|\text{tr}(D\varphi)| > 2$. The eigendirections of $D\varphi$ will then correspond to the expanding and contracting foliations of $\varphi$, but now the measures on these foliations are twisted by $\rho$. That is, the transverse measures on these foliations coming from $q$ can be understood locally, but globally will scale by the holonomy $\rho$ as we move around loops on the surface. In this way, realizing pseudo-Anosov maps as affine automorphisms of dilation surfaces helps us find twisted measured foliations, the set of which we denote $\mathcal{MF}_\rho$, that are preserved up to scale by the map. 

Given a pseudo-Anosov map $\psi$, we may also want to know if it actually is underlying topological map of an affine automorphism of some diation surface. If $\psi$ is a pseudo-Anosov map, then a necessary condition for $\psi$ to be realized as an affine automorphism of a dilation surface with holonomy $\rho$ is that $\psi^* \rho = \rho$. The following theorem, with slightly modified language, is proved by McMullen in \cite{M2}: 

\begin{thm}[McMullen] Given any pseudo-Anosov map $\psi: \Sigma \rightarrow \Sigma$ and compatible holonomy $\rho$, $\psi$ preserves up to scale a unique pair of twisted measured foliations in $\mathcal{MF}_\rho$. The underlying topological foliations correspond to the (untwisted) stable and unstable foliations of $\rho$.  
\end{thm}

This theorem essentially shows that pseudo-Anosov maps $\psi$ are realizable as $\varphi \in \Aff(X,\rho,q)$ when $\psi^* \rho = \rho$, since one can use a similar construction as in the untwisted case above to create a twisted quadratic differential out of the two elements of $\mathcal{MF}_\rho$ that are fixed up to scale by $\psi$. Then, the twisted quadratic differential $q \in \mathcal{QD}_\rho$ that we created gives a dilation surface on which $\varphi$ acts by an affine automorphism. 

\section{Twisted quadratic differentials coming from pairs of multicurves}
\label{sec:multicurves}

We will now move on from realizing single elements of the mapping class group as affine automorphisms of dilation surfaces to realizing certain subgroups of the mapping class group. In particular, we will generalize a construction of Thurston that realizes pairs of Dehn multitwists in the affine automorphism groups of translation surfaces. 

One reason we might want to do such a thing is that we want to understand if dilation surfaces with nontrivial holonomy representations can have lattice Veech groups, which is currently an open question. If we could realize ``large" subgroups of the mapping class group in $\Aff(X,\rho,q)$, then such surfaces could be candidate surfaces for having lattice Veech groups. 

Another reason to generalize Thurston's construction to dilation surface is that it is a first step to understanding when two twisted measured foliations can be pieced together to form a twisted quadratic differential. As discussed in Section \ref{subsec:pA}, there is a well-known homeomorphism from pairs of untwisted measured foliations on a surface (minus some set $\Delta$, known as the diagonal) to the space of quadratic differentials on that surface: $\mathcal{MF} \times \mathcal{MF} \backslash \Delta \rightarrow \mathcal{QD}$. A natural question to ask then is if there is a similar map $\mathcal{MF}_\rho \times \mathcal{MF}_\rho \backslash \Delta \rightarrow \mathcal{QD}_\rho$ from pairs of twisted measured foliations minus some diagonal set to the space of twisted quadratic differentials. The construction for pairs of multicurves that we will establish in this section is a step in the direction of resolving this problem.

\subsection{Thurston's construction and an example} The main result of this section states roughly that we can construct dilation surfaces with Dehn multitwists in their affine automorphism groups coming from a pair of filling multicurves, as long as the holonomy representation is compatible with these multitwists. This result is inspired by a construction of Thurston involving translation surfaces, which we will briefly recap here. For more details about Thurston's construction, please see Theorem 14.1 of \cite{FM}.  

Thurston's construction gives a way of constructing pseudo-Anosov maps. It starts with a surface $S$ and a pair of filling multicurves $\alpha$ and $\beta$ on $\Sigma$. That is, the complement of $\alpha \cup \beta$ is a collection of topological disks. The construction then produces a half-translation surface $(X,q)$ such that $T_\alpha$ and $T_\beta$, Dehn multitwists around $\alpha$ and $\beta$, are in the Veech group of $(X,q)$. Certain products of the $T_\alpha$ and $T_\beta$ maps then give pseudo-Anosov maps on $S$. 

We generalize Thurston's construction to dilation surfaces with the following theorem. 

\thurston

Before we prove this theorem, let us see it in action via a simple example where $\alpha$ and $\beta$ are single simple closed curves on a genus $2$ surface.  

The surface that we will start with is the following topological surface, with simple closed curves $\alpha$ and $\beta$ marked on the surface. We notice that $\alpha \cup \beta$ fills the surface as the complement of these two curves is two topological disks. Applying Thurston's construction, we can then find a translation surface with Dehn twists $T_\alpha$ and $T_\beta$ in its Veech group, as shown in Figure \ref{fig:thurston_octagon}.

\begin{figure}[ht]
	\begin{center}
		\begin{tikzpicture} [scale = 1.5]
			\draw[fill = {rgb,255:red,202; green,229; blue,255}] (0,0) -- (1,0) -- (1.71, 0.71) -- (1.71, 1.71) -- (1, 2.42) -- (0,2.42) -- (-0.71, 1.71) -- (-0.71, 0.71) -- (0,0); 
			\draw[thick, red, mid = {>}] (0,0) -- (1,0); 
			\draw[thick, teal, mid = {>>}] (1,0) -- (1.71,0.71); 
			\draw[thick, purple, mid = {|}] (1.71,0.71) -- (1.71,1.65);
			\draw[thick, purple, mid = {|}] (1.71,0.76) -- (1.71,1.71);
			\draw[thick, blue, mid = {|}] (1.71,1.71) -- (1,2.42); 
			\draw[thick, red, mid = {>}] (0,2.42) -- (1,2.42); 
			\draw[thick, teal, mid = {>>}] (-0.71,1.71) -- (0,2.42); 
			\draw[thick, purple, mid = {|}] (-0.71,0.71) -- (-0.71,1.66); 
			\draw[thick, purple, mid = {|}] (-0.71,0.76) -- (-0.71,1.71); 
			\draw[thick, blue, mid = {|}] (0,0) -- (-0.71,0.71); 
			
			\draw[dotted] (-0.355, 0.355) -- (1.71, 1.21); 
			\draw[dotted] (-0.71, 1.21) -- (1.355, 2.065); 
			
			\draw[dashed] (-0.355, 2.065) -- (0.5, 0); 
			\draw[dashed] (0.5, 2.42) -- (1.355, 0.355); 
			
			\node at (1.5, 0.2) {$\beta$}; 
			\node at (-0.45, 0.2) {$\alpha$}; 
			\draw[->] (1.9, 1.2) -- (3,1.2);
		\end{tikzpicture} 
		\begin{tikzpicture}[scale = 1.8]
			\draw[fill = {rgb,255:red,202; green,229; blue,255}] (0,0) -- (2,0) -- (2,2) -- (0,2) -- (0,0); 
			\draw[thick, red, mid = {>}] (0,0) -- (1,0); 
			\draw[thick, red, mid = {>}] (1,2) -- (2,2); 
			\draw[thick, teal, mid = {>>}] (1,0) -- (2,0); 
			\draw[thick, teal, mid = {>>}] (0,2) -- (1,2); 
			\draw[thick, blue, mid = {|}] (0,0) -- (0,1); 
			\draw[thick, blue, mid = {|}] (2,1) -- (2,2); 
			\draw[thick, purple, mid = {|}] (0,1) -- (0,1.95); 
			\draw[thick, purple, mid = {|}] (0,1.05) -- (0,2); 
			\draw[thick, purple, mid = {|}] (2,0) -- (2,0.95);
			\draw[thick, purple, mid = {|}] (2,0.05) -- (2,1);
			\draw[fill] (0,0) circle [radius = 0.025];
			\draw[fill] (1,0) circle [radius = 0.025];
			\draw[fill] (2,0) circle [radius = 0.025];
			\draw[fill] (2,1) circle [radius = 0.025];
			\draw[fill] (2,2) circle [radius = 0.025];
			\draw[fill] (1,2) circle [radius = 0.025];
			\draw[fill] (0,2) circle [radius = 0.025];
			\draw[fill] (0,1) circle [radius = 0.025];
			
			\draw[dotted] (0,0.5) -- (2,0.5); 
			\draw[dotted] (0,1.5) -- (2,1.5); 
			\draw[dashed] (0.5,0) -- (0.5,2); 
			\draw[dashed] (1.5,0) -- (1.5,2); 
			\node at (1.75, 0.62) {$\alpha$};
			\node at (0.4, 1.75) {$\beta$};
			
		\end{tikzpicture}
	\end{center}
	\caption{Thurston's construction applied to the topological surface on the left with the pair of filling curves $\alpha$ and $\beta$ results in the translation surface on the right with horizontal and vertical cylinders with core curves $\alpha$ and $\beta$.}
	\label{fig:thurston_octagon}
\end{figure}
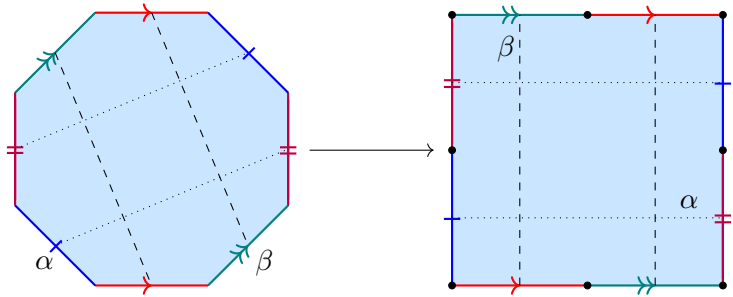

It is a genus two surface with one singularity of cone angle $2\pi$. Its Veech group contains the two parabolic elements $$\begin{bmatrix} 1 & 4 \\ 0 & 1 \end{bmatrix} \text{ and } \begin{bmatrix} 1 & 0 \\ -4 & 1 \end{bmatrix},$$ corresponding to Dehn twists around the dotted and dashed curves respectively, as seen in Figure \ref{fig:thurston_octagon}. Let $\alpha$ be the horizontal dotted curve and $\beta$ be the vertical dashed curve, and the two Dehn twists around these curves are then $T_\alpha$ and $T_\beta$. Any product of $T_\alpha$ and $T_\beta$ that has a trace with absolute value greater than $2$ would then be pseudo-Anosov. $T_\alpha (T_\beta)^{-1}$ is one such example. 

We wish now to extend this construction to dilation surfaces. In particular, for any holonomy $\rho : \pi_1(\Sigma) \rightarrow \mathbb{R}_+$ that is trivial on $\alpha$ and $\beta$, we wish to find a dilation surface $(X,\rho,q)$ with $T_\alpha$ and $T_\beta$ as standard Dehn twists on the surface. There is a $2$-dimensional space of such holonomies that should correspond to a $2$-dimensional space of dilation surfaces with $T_\alpha$ and $T_\beta$ in their affine automorphism group. 

To get this $2$-dimensional space of dilation surfaces, we imagine that the translation surface coming from Thurston's construction is made up of $4$ squares. We can then modify the lengths and widths of these squares to get a family of dilation surfaces with $T_\alpha$ and $T_\beta$ in their affine automorphism groups. We can fix the upper left square as a $1\times 1$ square and then modify the lengths $a$ and $b$ as shown in Figure \ref{fig:octagon_dilation} to be anything in $\mathbb{R}$. This gives us a $2$-dimensional family of dilation surfaces $(X, \rho, q)$ with $T_\alpha, T_\beta$ in their affine automorphism groups. One can then check that the associated maps $\rho$ vary over the $2$-dimensional space of holonomy representations that are trivial on the curves $\alpha$ and $\beta$. 

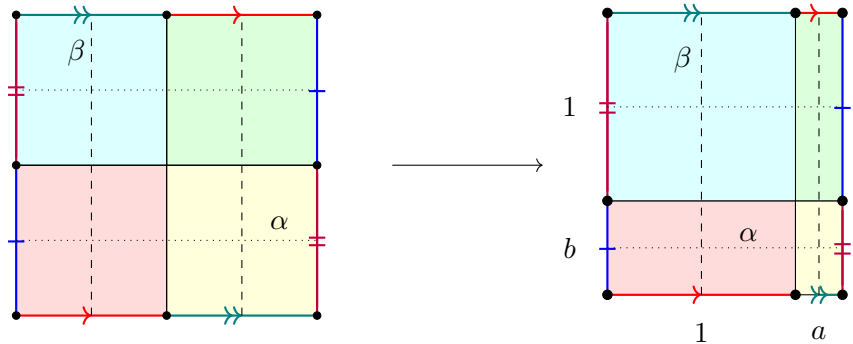
\begin{figure}[ht]
	\begin{center}
		\begin{tikzpicture}[scale = 2]
			\draw[fill = {rgb,255:red,255; green,220; blue,220}] (0,0) -- (1,0) -- (1,1) -- (0,1) -- (0,0); 
			\draw[fill = {rgb,255:red,255; green,255; blue,220}] (1,0) -- (2,0) -- (2,1) -- (1,1) -- (1,0); 
			\draw[fill = {rgb,255:red,220; green,255; blue,220}] (1,1) -- (2,1) -- (2,2) -- (1,2) -- (1,1); 
			\draw[fill = {rgb,255:red,220; green,255; blue,255}] (0,1) -- (1,1) -- (1,2) -- (0,2) -- (0,1);

			\draw[thick, red, mid = {>}] (0,0) -- (1,0); 
			\draw[thick, red, mid = {>}] (1,2) -- (2,2); 
			\draw[thick, teal, mid = {>>}] (1,0) -- (2,0); 
			\draw[thick, teal, mid = {>>}] (0,2) -- (1,2); 
			\draw[thick, blue, mid = {|}] (0,0) -- (0,1); 
			\draw[thick, blue, mid = {|}] (2,1) -- (2,2); 
			\draw[thick, purple, mid = {|}] (0,1) -- (0,1.95); 
			\draw[thick, purple, mid = {|}] (0,1.05) -- (0,2); 
			\draw[thick, purple, mid = {|}] (2,0) -- (2,0.95);
			\draw[thick, purple, mid = {|}] (2,0.05) -- (2,1);
			\draw[fill] (0,0) circle [radius = 0.025];
			\draw[fill] (1,0) circle [radius = 0.025];
			\draw[fill] (2,0) circle [radius = 0.025];
			\draw[fill] (2,1) circle [radius = 0.025];
			\draw[fill] (2,2) circle [radius = 0.025];
			\draw[fill] (1,2) circle [radius = 0.025];
			\draw[fill] (0,2) circle [radius = 0.025];
			\draw[fill] (0,1) circle [radius = 0.025];
			
			\draw[dotted] (0,0.5) -- (2,0.5); 
			\draw[dotted] (0,1.5) -- (2,1.5); 
			\draw[dashed] (0.5,0) -- (0.5,2); 
			\draw[dashed] (1.5,0) -- (1.5,2); 
			
			\draw (1,0) -- (1,2); 
			\draw (0,1) -- (2,1);
			\node at (1.75, 0.62) {$\alpha$};
			\node at (0.4, 1.75) {$\beta$};
			\draw[->] (2.5,1) -- (3.5,1); 
		\end{tikzpicture}
		\begin{tikzpicture}[scale = 2.5, baseline = -.4cm]
			\draw[fill = {rgb,255:red,255; green,220; blue,220}] (0,0) -- (1,0) -- (1,.5) -- (0,.5) -- (0,0); 
			\draw[fill = {rgb,255:red,255; green,255; blue,220}] (1,0) -- (1.25,0) -- (1.25,0.5) -- (1,0.5) -- (1,0); 
			\draw[fill = {rgb,255:red,220; green,255; blue,220}] (1,0.5) -- (1.25,0.5) -- (1.25,1.5) -- (1,1.5) -- (1,0.5); 
			\draw[fill = {rgb,255:red,220; green,255; blue,255}] (0,0.5) -- (1,0.5) -- (1,1.5) -- (0,1.5) -- (0,0.5);

			\draw[thick, red, mid = {>}] (0,0) -- (1,0); 
			\draw[thick, red, mid = {>}] (1,1.5) -- (1.25,1.5); 
			\draw[thick, teal, mid = {>>}] (1.1,0) -- (1.25,0); 
			\draw[thick, teal, mid = {>>}] (0,1.5) -- (1,1.5); 
			\draw[thick, blue, mid = {|}] (0,0) -- (0,.5); 
			\draw[thick, blue, mid = {|}] (1.25,0.5) -- (1.25,1.5); 
			\draw[thick, purple, mid = {|}] (0,0.5) -- (0,1.45); 
			\draw[thick, purple, mid = {|}] (0,0.55) -- (0,1.5); 
			\draw[thick, purple, mid = {|}] (1.25,0) -- (1.25,0.45);
			\draw[thick, purple, mid = {|}] (1.25,0.05) -- (1.25,0.5);
			
			\draw[fill] (0,0) circle [radius = 0.025];
			\draw[fill] (1,0) circle [radius = 0.025];
			\draw[fill] (1.25,0) circle [radius = 0.025];
			\draw[fill] (1.25,0.5) circle [radius = 0.025];
			\draw[fill] (1.25,1.5) circle [radius = 0.025];
			\draw[fill] (1,1.5) circle [radius = 0.025];
			\draw[fill] (0,1.5) circle [radius = 0.025];
			\draw[fill] (0,0.5) circle [radius = 0.025];
			
			\draw[dotted] (0,0.25) -- (1.25,0.25); 
			\draw[dotted] (0,1) -- (1.25,1); 
			\draw[dashed] (0.5,0) -- (0.5,1.5); 
			\draw[dashed] (1.125,0) -- (1.125,1.5); 
			
			\node at (0.75, 0.32) {$\alpha$};
			\node at (0.4, 1.25) {$\beta$};
			\node at (0.5, -0.2) {$1$}; 
			\node at (1.125, -0.2) {$a$}; 
			\node at (-0.2, 0.25) {$b$}; 
			\node at (-0.2, 1) {$1$}; 
		\end{tikzpicture} 
		\caption{Both surfaces have $T_\alpha$ and $T_\beta$ in their affine automorphism groups. The left surface has trivial holonomy $\rho$ whereas the right surface has a specific nontrivial holonomy $\rho$.}
		\label{fig:octagon_dilation}
	\end{center} 
\end{figure}

The basic idea of this example was to build our desired dilation surface out of rectangles so that the resulting surface had a horizontal and vertical cylinder corresponding to the $\alpha$ and $\beta$ curves and the appropriate holonomy representation $\rho$. We will expand on this idea in the next section when we discuss this construction in general, when $\alpha$ and $\beta$ may be multicurves.

\subsection{The construction}

In this section, we will prove Theorem \ref{thm:thurston} by way of more general propositions about building dilation surfaces out of pairs of filling multicurves. The proofs in this section will rely on a version of the Perron-Frobenius theorem, which we state here. A non-negative matrix $A$ is said to be \textbf{irreducible} if for every pair of indices $i$ and $j$, there exists a positive integer $k$ for which $A^k_{i,j} > 0$. 

\begin{thm}[Perron-Frobenius theorem] Let $A$ be an $n \times n$ non-negative irreducible matrix. Then, 
	\begin{itemize}
		\item[(i)] $A$ has an eigenvalue $\lambda_{pf} > 0$ of multiplicity $1$ with a positive eigenvector $\mathbf{v}_{pf}$, called the \textbf{Perron-Frobenius eigenvalue} and \textbf{Perron-Frobenius eigenvector} respectively. 
		\item[(ii)] any other eigenvalue $\lambda$ of $A$ satisfies $|\lambda| < \lambda_pf$. 
	\end{itemize}
\end{thm}

We will begin now with a proposition showing that from the data of a pair of multicurves $\alpha$ and $\beta$, and a holonomy representation $\rho$ that is trivial on each curve, we can build families of quadratic differentials with holonomy $\rho$ and the curves of $\alpha$ and $\beta$ as the core curves of cylinders on the surface.  

We first notice that if $\rho : \pi_1(S) \rightarrow \R_+$ is trivial on a curve $c$ on the surface, then the flat complex line bundle $L_\rho$ restricts to a trivial bundle over the curve $c$. The vector space of real flat sections of $L_\rho|_{c}$ is then isomorphic to $\R$. With this, we then have the following proposition. 

\begin{prop} 
	\label{prop:qd}
	Let $\alpha = \{\alpha_1, \ldots, \alpha_k\}$ and $\beta = \{ \beta_1, \ldots, \beta_l\}$ be two multicurves on a surface $S$ such that $\mathcal{C} := \alpha \cup \beta$ fills $S$. If $\rho : \pi_1(S) \rightarrow \R_+$ is trivial on each curve $\alpha_i \in \alpha$ and $\beta_j \in \beta$, then the set of twisted quadratic differentials with holonomy $\rho$ and multicurves $\alpha$ and $\beta$ as the core curves of a complete horizontal and vertical cylinder decomposition is in direct correspondence with $(k+l)$-tuples of positive real sections of $L_\rho$ restricted to each $\alpha_i$ and $\beta_j$. 
\end{prop} 

\begin{proof} Let $L_\rho$ be the flat complex line bundled associated to the holonomy representation $\rho$. Then, since the holonomy around any curve in $\mathcal{C}$ is trivial, the restriction of $L_\rho$ to any curve $c_i \in \mathcal{C}$ is trivial. Associated to each $c_i$, we can define $V_i \cong \R$ to be the vector space of real flat sections of $L_\rho|_{c_i}$. Now suppose that we have a positive vector $(v_1, \ldots, v_{k+l}) \in (\R_+)^{k+l}$, where each $v_i \in V_i$ is a real flat section. We're going to create a dilation surface that has a Euclidean cylinder for each $c_i$. The $v_i$'s will determine the ``widths" of the cylinders. We note here that because of the holonomy $\rho$, the widths of cylinders are not defined globally. 
	
	 Let $v_{(\beta,p_j)}$ be the $v_i$ that corresponds to the curve in $\beta$ that goes through point $p_j$. We define $v_{(\alpha, p_j)}$ analogously for $\alpha$ curves. For each point $p_i$ that is an intersection point of an $\alpha$ curve and a $\beta$ curve, the line bundle $L_\rho$ gives us isomorphisms $T_{p_i, \alpha\beta} : V_{(p_i, \alpha)} \rightarrow V_{(p_i,\beta)}$ and $T_{p_i, \beta\alpha} : V_{(p_i, \beta)} \rightarrow V_{(p_i,\alpha)}$ between the vector spaces associated to the $\alpha$ and $\beta$ curves intersecting at the point $p_i$. Then locally, the ratio of the widths of the $\beta$ and $\alpha$ cylinders intersecting at a point $p_i$ is given by the ratio of $v_{(p_i, \beta)}$ to $T_{p_i, \alpha \beta} (v_{(p_i, \alpha)})$.

	Now, to build a quadratic differential with holonomy $\rho$, we will construct a rectangle for each intersection point $p_i$. The width of the rectangle will be $v_{(p_i, \beta)}$ and the height will be $T_{p_i, \alpha \beta} (v_{(p_i, \alpha)})$. These rectangles can then be glued together by dilation with gluing pattern specified by the intersection pattern of the $\alpha$ and $\beta$ curves. Since there was no holonomy around any curve $c_i$, the rectangles along the curve $c_i$ will always glue together to form a Euclidean cylinder. 
	
	Conversely, if we had started with a twisted quadratic differential with holonomy $\rho$ that completely decomposed into horizontal and vertical cylinders with core curves the multicurves $\alpha$ and $\beta$, we can find the widths of these cylinders as flat sections of the line bundles $L_\rho$ restricted to the curves $\alpha_i$ and $\beta_j$. 
\end{proof}

The above proposition can be thought of as a special case of a generalization of the well-known fact that two transverse (untwisted) measured laminations can be combined in a unique way to form a unique quadratic differential. Here, we are taking two transverse twisted measured laminations coming from multicurves and constructing a unique twisted quadratic differential with these twisted measured laminations as its horizontal and vertical laminations. 

We can actually do a bit more with this construction and specify up to scale the moduli of the cylinders corresponding to the $\alpha$ and $\beta$ curves. 

\begin{prop} 
	\label{prop:moduli}
	Let $\alpha = \{\alpha_1, \ldots, \alpha_k\}$ and $\beta = \{ \beta_1, \ldots, \beta_l\}$ be two multicurves on a surface $S$ such that $\mathcal{C} := \alpha \cup \beta$ fills $S$. If $\rho : \pi_1(S) \rightarrow \R_+$ is trivial on each $c_i \in \mathcal{C}$, then for any list of positive real numbers $m_1, \ldots, m_{k+l}$, there exists a $\lambda > 0$ and a twisted quadratic differential with holonomy $\rho$ that splits completely into horizontal and vertical cylinders with core curves $\alpha_i$ and $\beta_j$ respectively, and moduli $\lambda m_i$ for each cylinder $c_i$. 
\end{prop} 

\begin{proof}
	As in the proof of Proposition \ref{prop:qd}, we can define vector spaces $V_i$ of real flat sections of the line bundles $L_\rho|_{c_i}$ where $c_i \in \mathcal{C} := \alpha \cup \beta$. Then, as before, choosing $v_i \in V_i$ gives us dimensions for rectangles corresponding to each point $p_i$ in the intersection of the $\alpha$ and $\beta$ curves. Then, for each curve $c_i \in \alpha$, the modulus $m_i$ of the corresponding cylinder satisfies 
	$$\frac{1}{m_i} = \frac{\text{circumference}}{\text{height}} = \frac{1}{v_i} \sum_{p_j \in c_i \cap \beta} T_{p_j, \beta\alpha} (v_{(\beta, p_j)}),$$ where the sum is taken over all of the intersection points of the curve $c_i$ with curves in $\beta$. $1/m_i$ for $c_i$ a curve in $\beta$ is then defined similarly, but with the roles of $\beta$ and $\alpha$ reversed. 
	
	Rewriting the above equation, we get that 
	$$m_i\sum_{p_j \in c_i \cap\beta} T_{p_j, \beta\alpha} (v_{(\beta, p_j)}) = v_i = v_{(\alpha, p_j)}.$$ 
	
	We can define a linear operator $U : \prod_{i=1}^{k+l} V_i \rightarrow \prod_{i=1}^{k+l} V_i$ by letting 
	$$U_i = \begin{cases} m_i\sum_{p_j \in c_i \cap \beta} T_{p_j, \beta\alpha} (v_{(\beta, p_j)}) , & c_i \in \alpha \\ m_i\sum_{p_j \in c_i \cap \alpha} T_{p_j, \alpha\beta} (v_{(\alpha, p_j)}) , & c_i \in \beta. \end{cases}$$ 
	
	Thinking of $U$ as a matrix, $U$ is then non-negative, and $U_{i,j}$ is positive if and only if the curve $c_i$ intersects the curve $c_j$. Since the curves $\mathcal{C}$ fill the surface, the matrix $U$ is irreducible. Therefore, the Perron-Frobenius theorem applies and $U$ has a positive eigenvalue $\lambda_{pf} > 1$ and a corresponding positive eigenvector $\textbf{v}_{pf}$. This $\textbf{v}_{pf}$ is then a vector of widths that give us cylinders of moduli $m_i / \lambda_{pf}$. So we can construct a twisted quadratic differential with holonomy $\rho$ and moduli $\lambda m_i$ where $\lambda = 1/\lambda_{pf}$. 
\end{proof}

Let $\alpha$ and $\beta$ be multicurves that fill a surface $S$. One way to create a dilation surface with the pseudo-Anosov map $\psi$ represented in its affine automorphism group is to create a surface where the Dehn multitwists $T_\alpha$ and $T_\beta$ are still in the affine automorphism group. A necessary condition is that the holonomy representation of the dilation surface is compatible with $\psi$. If $\Sigma$ is the underlying topological surface to $(X,\rho,q)$, then let $\rho : \pi_1(\Sigma) \rightarrow \R_+$ be a holonomy representation that sends the homotopy classes of the curves $\alpha$ and $\beta$ both to $1$, so that there is trivial holonomy around these two curves. This holonomy condition will guarantee that our Dehn multitwists will be standard Dehn multitwists.

Now we can prove the main theorem of this section, Theorem \ref{thm:thurston}

\begin{proof}[Proof of Theorem \ref{thm:thurston}] This theorem is really a corollary of Proposition \ref{prop:moduli}, and follows by specifying that the moduli of the cylinders corresponding to $\alpha$ and $\beta$ are all the same. We can let the $m_i$'s in the previous proposition be all equal to $1$. Then, the constructed quadratic differentials has Dehn twists $T_\alpha$ and $T_\beta$ in its affine automorphism group, with derivatives given by $$T_\alpha \mapsto \begin{bmatrix} 1 & \mu \\ 0 & 1 \end{bmatrix} \text{ and } T_\beta \mapsto \begin{bmatrix} 1 & 0 \\ -\mu & 1 \end{bmatrix},$$ where $\mu = 1/\lambda_{pf}$.  
\end{proof} 




One might wonder if this modified Thurston construction could produce dilation surfaces with nontrivial holonomy that have a lattice Veech group. After all, there exist translation surfaces produced from the standard Thurston construction such that $T_\alpha$ and $T_\beta$ generate a lattice in the Veech group. Could we take one of these examples and perturb it slightly via this modified Thurston construction to get a dilation surface such that $T_\alpha$ and $T_\beta$ generate a lattice subgroup of $\text{SL}(X, \rho, q)$? 

The answer to this question is unfortunately no. In \cite{L}, Leininger gives a complete list of curve systems $\alpha$ and $\beta$ such that the Thurston construction produces a translation surface where $T_\alpha$ and $T_\beta$ generate a lattice in $\text{SL}(X, q)$. It is then not hard to check that in each of these examples, the curves in $\alpha$ and $\beta$ generate all of the homology of the surface and thus the modified Thurston construction does not produce any dilation surfaces with nontrivial holonomy from these systems. 

While it is still possible that dilation surfaces with nontrivial holonomy and lattice Veech groups exist, and even that the modified Thurston construction could produce such examples (if the Veech group includes more than just the $T_\alpha$ and $T_\beta$ maps), more work would be needed to find these examples.

\subsection{Examples}

Our first example of the construction described in Theorem \ref{thm:thurston} was given in an earlier section and illustrated in Figure \ref{fig:octagon_dilation}. We note here that for certain pseudo-Anosov maps $\psi$ like $\psi = T_\alpha (T_\beta)^{-1}$ in that example, the process described in Theorem \ref{thm:thurston} actually generates a family of dilation surfaces $(X, \rho, q)$ with $\psi$ represented in the affine automorphism group for the whole set of holonomies $\rho$ compatible with $\psi$. This is made precise in the following proposition. 

\begin{prop} If $\alpha$ and $\beta$ are simple closed curves that together fill the surface and have algebraic intersection number $i(\alpha, \beta) \neq 0$, then the set of holonomies compatible with $\psi := T_\alpha (T_\beta)^{-1}$ is exactly the set of holonomies that are trivial on both $\alpha$ and $\beta$.  
\end{prop} 
\begin{proof} Let $\rho : \pi_(X) \rightarrow \R$ be a holonomy representation. Then, since $$\psi_* [\alpha] = [\alpha] - i(\alpha, \beta)[\beta],$$ for $\rho$ to be compatible with $\psi$, we want that $\rho([\alpha] - i(\alpha, \beta)[\beta]) = \rho(\psi_*[\alpha]) = \rho([\alpha])$. Thus, we must have that $i(\alpha, \beta) \rho([\beta)] = 0$ and so  $\rho([\beta]) = 0$ since $i(\alpha, \beta) \neq 0$. We also have that 
	$$\psi_* [\beta] = [\beta] + i(\beta, \alpha) [\alpha] - i(\beta, \alpha) i (\alpha,\beta) [\beta].$$ 
	Since $\rho([\beta]) = 0$, by similar reasoning as before, we have that $\rho([\alpha])$ must also be zero. 
\end{proof} 

In other words, for $\alpha$ and $\beta$ satisfying the conditions of the proposition above, our modified Thurston construction gives another way of constructing the whole family of dilation surfaces $(X, \rho, q)$ with $\psi = T_\alpha (T_\beta)^{-1}$ in $\Aff(X, \rho, q)$, where this family ranges over all $\rho$ that are compatible with $\psi$. 

We can also construct examples where we begin with $\alpha$ and $\beta$ that are not just simple closed curves. Let us consider the example shown in Figure \ref{fig:L_example}. This is a genus $2$ example where the simple closed curve $\alpha$ and the multicurve $\beta = \{\beta_1, \beta_2\}$ fill the surface. The homology of the surface is generated by the four curves $\gamma_1, \ldots, \gamma_4$ as shown below. 

\begin{figure}[ht]
\begin{center}
	\begin{tikzpicture}[scale = 1.5]
		\draw[fill = {rgb,255:red,202; green,229; blue,255}] (0,0) -- (2,0) -- (2,1) -- (1,1) -- (1,2) -- (0,2) -- (0,0); 
		\draw[thick, red, mid={|}] (0,0) -- (1,0); 
		\draw[thick, blue, mid={|}] (1,0) -- (1.95,0); 
		\draw[thick, blue, mid={|}] (1.05,0) -- (2,0); 
		\draw[thick, purple, mid ={>}] (2,0) -- (2,1); 
		\draw[thick,blue, mid = {|}] (1.95,1) -- (1,1); 
		\draw[thick,blue, mid = {|}] (2,1) -- (1.05,1); 
		\draw[thick,orange, mid={>>}] (1,1) -- (1,2);
		\draw[thick,red,mid={|}] (1,2) -- (0,2);
		\draw[thick,purple,mid={>}] (0,1) -- (0,2);
		\draw[thick,orange, mid={>>}] (0,0) -- (0,1);  
		\draw[fill] (0,0) circle [radius = 0.05];
		\draw[fill] (1,0) circle [radius = 0.05];
		\draw[fill] (2,0) circle [radius = 0.05];
		\draw[fill] (0,1) circle [radius = 0.05];
		\draw[fill] (0,2) circle [radius = 0.05];
		\draw[fill] (1,1) circle [radius = 0.05];
		\draw[fill] (2,1) circle [radius = 0.05];
		\draw[fill] (1,2) circle [radius = 0.05];
		\draw[dotted] (0,0.5) -- (2,0.5);
		\draw[dotted] (0,1.5) -- (1,1.5); 
		\draw[dashed] (0.5,0) -- (0.5, 2);
		\draw[dashed] (1.5,0) -- (1.5, 1); 
		\node at (0.5,-0.35) {$\beta_1$};
		\node at (1.5,-0.35) {$\beta_2$};
		\node at (-0.25, 0.5) {$\alpha$}; 
	\end{tikzpicture} 
\hspace{1cm}
	\begin{tikzpicture}[scale = 1.5]
	\draw[fill = {rgb,255:red,202; green,229; blue,255}] (0,0) -- (2,0) -- (2,1) -- (1,1) -- (1,2) -- (0,2) -- (0,0); 
	\draw[thick, red, mid={|}] (0,0) -- (1,0); 
	\draw[thick, blue, mid={|}] (1,0) -- (1.95,0); 
	\draw[thick, blue, mid={|}] (1.05,0) -- (2,0); 
	\draw[thick, purple, mid ={>}] (2,0) -- (2,1); 
	\draw[thick,blue, mid = {|}] (1.95,1) -- (1,1); 
	\draw[thick,blue, mid = {|}] (2,1) -- (1.05,1); 
	\draw[thick,orange, mid={>>}] (1,1) -- (1,2);
	\draw[thick,red,mid={|}] (1,2) -- (0,2);
	\draw[thick,purple,mid={>}] (0,1) -- (0,2);
	\draw[thick,orange, mid={>>}] (0,0) -- (0,1);  
	\draw[fill] (0,0) circle [radius = 0.05];
	\draw[fill] (1,0) circle [radius = 0.05];
	\draw[fill] (2,0) circle [radius = 0.05];
	\draw[fill] (0,1) circle [radius = 0.05];
	\draw[fill] (0,2) circle [radius = 0.05];
	\draw[fill] (1,1) circle [radius = 0.05];
	\draw[fill] (2,1) circle [radius = 0.05];
	\draw[fill] (1,2) circle [radius = 0.05];
	
	\draw[dotted, m3 = {>}, red] (0.5,0) -- (0.5, 2);
	\draw[dotted, m1 = {>}, blue] (1.5,0) -- (1.5, 1); 
	\draw[dotted, mid = {>}, purple] (0,1.25) -- (2,0.25); 
	\draw[dotted, m1 = {>}, orange] (0,0.25) -- (1,1.25); 
	\node[red] at (0.5,-0.35) {$\gamma_1$};
	\node[blue] at (1.5,-0.35) {$\gamma_2$};
	\node[orange] at (-0.25, 0.25) {$\gamma_3$}; 
	\node[purple] at (-0.25, 1.25) {$\gamma_4$}; 
\end{tikzpicture} 
\end{center}
\caption{A genus $2$ surface where the curve $\alpha$ and multicurve $\beta = \{\beta_1, \beta_2\}$ fill the surface. The homology of the surface is generated by the four curves $\gamma_1, \ldots, \gamma_4$ as shown on the right.}
\label{fig:L_example}
\end{figure}
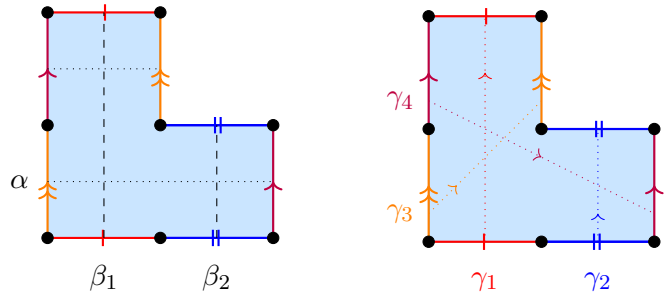

Then, the set of holonomy representations $\rho : \pi_1(X) \rightarrow \R_+$ that are trivial on each of $\alpha, \beta_1, \beta_2$ is a one dimensional vector space consisting of those $\rho : (\gamma_1, \gamma_2, \gamma_3, \gamma_4) \mapsto (1,1, a, 1/a)$ for some $a \in \R$. We expect a unique (up to scale) dilation surface with Dehn multitwists $T_\alpha$ and $T_\beta$ around cylinders of the same modulus in its affine automorphism group for each of these holonomy representations. 

After normalizing one length to $1$, we may assume that our surface has widths for its horizontal and vertical cylinders as labeled in Figure \ref{fig:L_example_lengths}.

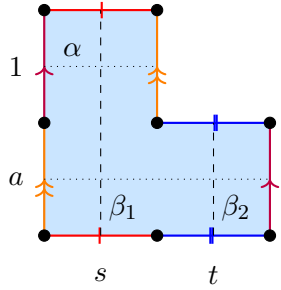
\begin{figure}[ht]
\begin{center}
	\begin{tikzpicture}[scale = 1.5]
		\draw[fill = {rgb,255:red,202; green,229; blue,255}] (0,0) -- (2,0) -- (2,1) -- (1,1) -- (1,2) -- (0,2) -- (0,0); 
		\draw[thick, red, mid={|}] (0,0) -- (1,0); 
		\draw[thick, blue, mid={||}] (1,0) -- (2,0); 
		\draw[thick, purple, mid ={>}] (2,0) -- (2,1); 
		\draw[thick,blue, mid = {||}] (2,1) -- (1,1); 
		\draw[thick,orange, mid={>>}] (1,1) -- (1,2);
		\draw[thick,red,mid={|}] (1,2) -- (0,2);
		\draw[thick,purple,mid={>}] (0,1) -- (0,2);
		\draw[thick,orange, mid={>>}] (0,0) -- (0,1);  
		\draw[fill] (0,0) circle [radius = 0.05];
		\draw[fill] (1,0) circle [radius = 0.05];
		\draw[fill] (2,0) circle [radius = 0.05];
		\draw[fill] (0,1) circle [radius = 0.05];
		\draw[fill] (0,2) circle [radius = 0.05];
		\draw[fill] (1,1) circle [radius = 0.05];
		\draw[fill] (2,1) circle [radius = 0.05];
		\draw[fill] (1,2) circle [radius = 0.05];
		\draw[dotted] (0,0.5) -- (2,0.5);
		\draw[dotted] (0,1.5) -- (1,1.5); 
		\draw[dashed] (0.5,0) -- (0.5, 2);
		\draw[dashed] (1.5,0) -- (1.5, 1); 
		\node at (0.5,-0.35) {$s$};
		\node at (1.5,-0.35) {$t$};
		\node at (-0.25, 0.5) {$a$}; 
		\node at (-0.25, 1.5) {$1$}; 
		\node at (0.7, 0.25) {$\beta_1$}; 
		\node at (1.7, 0.25) {$\beta_2$};
		\node at (0.25, 1.65) {$\alpha$};  
	\end{tikzpicture} 
\end{center}
\caption{Our example surface with some lengths labeled.}
\label{fig:L_example_lengths}
\end{figure}

Then, the moduli of the $\beta_1$ and $\beta_2$ curves are $s/(1+a)$ and $t/a$. For these moduli to be equal, we must have that $t = \left(\frac{a}{1+a}\right) s$. The modulus of the $\alpha$ cylinder is then $\left(s + \frac{s}{a} + \frac{s}{1+a}\right)^{-1}$. For this modulus to be equal to the moduli of $\beta_1$ and $\beta_2$, we then need that $s = \frac{\sqrt{a}(1+a)}{\sqrt{a^2 + 3a+1}}$. Thus, there exists a unique dilation surface with holonomy $\rho$, up to scaling by $\R_+$, with $T_\alpha$ and $T_\beta$ in the affine automorphism group as Dehn multitwists around cylinders of the same modulus.

\section{Further Questions}

There are still many natural questions related to the realization problem that are still unanswered. We end with some of these questions: 
\begin{enumerate}
	\item If $\psi \in \text{Mod}(\Sigma)$ satisfies that $\psi^n$ is a Dehn multitwist for some $n$, and $\rho : \pi_1(\Sigma) \rightarrow \mathbb{R}_+$ is a compatible holonomy representation, can $\psi$ necessarily be realized as an affine automorphism of some dilation surface $(X, \rho, q)$? 
	\item Given a mapping class group $\psi$ and compatible holonomy $\rho$ for which there exists a dilation surface $(X, \rho, q)$ where $\psi$ is represented by an affine automorphism, what can we say about the uniqueness of the surface $(X, \rho, q)$? 
	\item Do any of the dilation surfaces coming from pairs of Dehn multitwist construction in Theorem \ref{thm:thurston} had a lattice Veech group? While the pairs of Dehn multitwists are never enough to generate a lattice subgroup of $\text{SL}(X, \mathbb{R})$, perhaps there are enough other elements in their Veech groups so that the whole Veech group is a lattice. 
	\item More generally, are there dilation surfaces with nontrivial holonomy that have a lattice Veech group? 
	\item Are there dilation surfaces with Veech groups that are large in other ways? For example, are there dilation surfaces with infinitely generated Veech groups? 
\end{enumerate}

\bibliographystyle{alpha}
\bibliography{references}

\end{document}